\documentclass{./MyPreprintStyleFile}

\begin{document}

\title{On Morita Equivalences Between KLR Algebras and VV Algebras}
\author{Ruari Walker}
\address{Institute of Algebra and Number Theory\\  University of Stuttgart \\
Pfaffenwaldring 57 \\ 70569 Stuttgart\\ Germany}
\email{ruari.walker@uni-stuttgart.mathematik.de}
\maketitle

\thispagestyle{empty}

\begin{abstract}
This paper is investigative work into the properties of a family of graded algebras recently defined by Varagnolo and Vasserot, which we call VV algebras. We compare categories of modules over KLR algebras with categories of modules over VV algebras, establishing various Morita equivalences. Using these Morita equivalences we are able to prove several properties of certain classes of VV algebras such as (graded) affine cellularity and affine quasi-heredity.
\end{abstract}

\section*{Introduction}
Varagnolo and Vasserot have recently defined a new family of graded algebras, the representation theory of which is closely related to the representation theory of the affine Hecke algebras of type B. Indeed, they prove in \cite{VaragnoloVasserot} that categories of finite-dimensional modules over these algebras are equivalent to categories of finite-dimensional modules over affine Hecke algebras of type $B$, $\mathrm{H}^B_n$. They also use these algebras to prove a conjecture of Enomoto and Kashiwara which states that the representations of the affine Hecke algebra of type $B$ categorify a simple highest weight module for a certain quantum group (see \cite{EnomotoKashiwara2}). These are the main motivating reasons behind studying these algebras. Throughout this work we refer to these algebras as VV algebras. One of the advantages of working with VV algebras is that they have a non-trivial grading, whilst affine Hecke algebras of type $B$ do not. Similarly to KLR algebras, VV algebras depend upon a quiver $\Gamma$, which now comes equipped with an involution $\theta$, and a dimension vector $\nu$ which is invariant under this involution. The vertices of $\Gamma$ are labelled by an indexing set $I$ which is an orbit arising from a $\mathbb{Z}\rtimes \mathbb{Z}_2$-action on the ground field. They also depend upon two non-zero elements of the field, $p$ and $q$, which correspond to the deformation parameters of $\mathrm{H}^B_n$. It turns out that there are different cases to consider when studying VV algebras; 
\begin{equation*}
\begin{array}{ll}
\textrm{\textbf{(A1)}}& p,q \not\in I \\
\textrm{\textbf{(A2)}}& q \in I, p \not\in I \\
\textrm{\textbf{(A3)}}& p \in I, q \not\in I
\end{array}
\end{equation*}
One can also consider the case where both $p,q \in I$ which we do not treat here. 
\\
\\
In this paper we compare categories of modules over KLR algebras with categories of modules over VV algebras. In case (A1) we find that these module categories are indeed equivalent. In other cases we find categories of modules over VV algebras are equivalent to categories of modules over the tensor product of KLR algebras with the path algebra of a certain quiver. Using these Morita equivalences, many properties of KLR algebras, and the given path algebra, can be transferred to classes of VV algebras. In particular, the main result of this work is that certain subclasses of VV algebras are affine cellular and even affine quasi-hereditary.
\begin{maintheorem}[\ref{corollary: VV aff qh when q in I}, \ref{corollary: VV aff qh when p,q not in I}, \ref{corollary:VV algs q in I are affine cellular}, \ref{corollary:VV are aff cellular when p,q not in I}] Let $\nu \in {^\theta}\mathbb{N}I$. If $p$ is not a root of unity and we are either in case (A1) or case (A2) with the additional assumption that $\nu$ has multiplicity one, the algebras $\vv$ are (graded) affine cellular and affine quasi-hereditary.
\end{maintheorem}
In Section \ref{chapter:KLR algebras and VV algebras}, we define KLR algebras of type $A$ and the VV algebras, which are our main objects of study, and prove a proposition which shows how these algebras are related. Namely, under certain conditions, KLR algebras are idempotent subalgebras of VV algebras. That is, there is an algebra isomorphism $\textbf{e}\vv\textbf{e} \cong \subklr$, for a certain $\tilde{\nu} \in \mathbb{N}I$ and an idempotent $\textbf{e} \in \vv$. 
\\
\\
Section \ref{chapter:MEs between KLR and VV} starts with why we may assume the dimension vector $\nu$ defining the VV algebra $\vv$ has connected support. 
\begin{theorem*} [\ref{theorem:ME in the separated case}]
Let $I,J$ be two $\mathbb{Z} \rtimes \mathbb{Z}_2$-orbits such that $I \cap J = \emptyset$. Then, for $\nu = \nu_1 + \nu_2 \in {^\theta}\mathbb{N}(I \cup J)$ with $\nu_1 \in {^\theta}\mathbb{N}I$ and $\nu_2 \in {^\theta}\mathbb{N}J$, $\vv$ and $\w_{\nu_1} \otimes_{\textbf{k}} \w_{\nu_2}$ are Morita equivalent.
\end{theorem*}
We then consider various classes of VV algebras and compare their module categories with module categories over KLR algebras. In particular, in the case $p,q \not\in I$ we find Morita equivalence between VV algebras and KLR algebras.
\begin{theorem*} [\ref{theorem:VV Morita equiv to KLR}]
In case (A1), the algebras $\vv$ and $\subklr$ are Morita equivalent.
\end{theorem*}
Section \ref{chapter:MEs between KLR and VV} also introduces the notions of affine cellularity, as defined by Koenig and Xi in \cite{KoenigXi}, and affine quasi-heredity as defined by Kleshchev in \cite{KleshchevAffineQH}. We then prove some statements about affine quasi-hereditary algebras and show, in the case $q \in I$, $p \not\in I$, when $\nu$ has multiplicity one, that the VV algebras are (graded) affine cellular and affine quasi-hereditary. To do this we prove the following result, where $A$ and $\tilde{A}$ are the path algebras of a given quiver.
\begin{theorem*} [\ref{theorem:ME in case q in I mult one}]
If $p$ not a root of unity and we are in the case (A2) with $\nu$ of multiplicity one, the algebras $\vv$ and $A^{\otimes (m-1)}\otimes \tilde{A}$ are Morita equivalent.
\end{theorem*}
At the end of Section \ref{chapter:MEs between KLR and VV} we show that when we relax this multiplicity condition imposed on $\nu$ we obtain the following Morita equivalence.
\begin{theorem*} [\ref{theorem:W and R x A are ME}]
In case (A2), when $q$ has multiplicity one in $\nu$, the algebras $\vv$ and $\subklr^+ \otimes_{\textbf{k}[z]} \tilde{A}$ are Morita equivalent.
\end{theorem*}
At the end of Section \ref{chapter:MEs between KLR and VV} we prove a Morita equivalence statement in the case $p \in I$, $q \not\in I$, $p$ not a root of unity, where $p$ has multiplicity exactly two in $\nu$.
\begin{theorem*} [\ref{theorem:ME when p in I}]
If $p$ is not a root of unity and we are in the case (A3), when $p$ has multiplicity two in $\nu$, the algebras $\vv$ and $\subklr^+ \otimes_{\textbf{k}[z]} \tilde{A}$ are Morita equivalent.
\end{theorem*}
\subsection*{Acknowledgements} This work makes up part of my PhD thesis which was completed at the University of East Anglia under the supervision of Dr Vanessa Miemietz. I would like to thank Vanessa for all her help, guidance and ideas throughout. I also thank the School of Mathematics and the University of East Anglia for the studentship and the financial support that came with it. This work was submitted whilst at the University of Stuttgart. I would like to thank Prof. Steffen Koenig for allowing me to complete this work in Stuttgart and for many helpful and interesting discussions.
\section{KLR Algebras and VV Algebras}  \label{chapter:KLR algebras and VV algebras}
\subsection{Preliminaries}
Throughout this work $\mathbf{k}$ will denote a field with characteristic not equal to 2. Any grading will always mean a $\mathbb{Z}$-grading. We start by introduce the notion of graded dimension for a graded module over a $\textbf{k}$-algebra $A$. We write $q$ to be both a formal variable and a degree shift functor, shifting the degree by 1. That is, for a graded $A$-module $M=\bigoplus_{n \in \mathbb{Z}} M_n$, $qM$ is again a graded $A$-module with $(qM)_k=M_{k-1}$. The graded $A$-module $M$ is said to be locally finite-dimensional if each graded component $M_k$ is finite-dimensional. In this case the graded dimension of $M$ is defined to be 
\begin{equation*}
\textrm{dim}_q M = \sum_{n \in \mathbb{Z}} (\textrm{dim}(M_n)) q^n
\end{equation*}
where dim$(M_n)$ is the dimension of $M_n$ over $\textbf{k}$.
\begin{example} \label{example:graded dim of poly. ring}
Let $\textbf{k}$ be a field and consider the polynomial ring $\textbf{k}[x]$. As a $\textbf{k}$-vector space, $\textbf{k}[x]$ has a basis $\{ 1,x,x^2,x^3, \ldots \}$. Each graded component is of dimension 1. Then, 
\begin{equation*}
\textrm{dim}_q\textbf{k}[x]=\sum_{n \in \mathbb{Z}_{\geq 0}} q^n=\frac{1}{1-q}.
\end{equation*}
\end{example}
\subsection{KLR Algebras} \label{klr}
In this section we define a family of algebras which were introduced by Khovanov, Lauda and independently by Rouquier. They are known as KLR algebras, or sometimes as quiver Hecke algebras.
\\
\\
Fix an element $p \in \mathbf{k}^{\times}$. Define an action of $\mathbb{Z}$ on $\mathbf{k}^{\times}$ as follows,
\begin{equation*}
n \cdot \lambda = p^{2n}\lambda.
\end{equation*}
Let $\tilde{I}$ be a $\mathbb{Z}$-orbit. So $\tilde{I}=\tilde{I}_\lambda$ is the $\mathbb{Z}$-orbit of $\lambda$,
\begin{equation*}
\tilde{I}=\tilde{I}_\lambda=\{ p^{2n}\lambda \mid n \in \mathbb{Z} \}.
\end{equation*}
To $\tilde{I}$ we associate a quiver $\tilde{\Gamma}=\tilde{\Gamma}_{\tilde{I}}$. The vertices of $\tilde{\Gamma}$ are the elements $i\in \tilde{I}$ and we have arrows $p^2i \longrightarrow i$ for every $i\in \tilde{I}$. We always assume that $\pm 1 \notin \tilde{I}$ and that $p \neq \pm 1$.
\\
\\
Now define $\mathbb{N}\tilde{I}=\{ \tilde{\nu} = \sum_{i\in \tilde{I}}\tilde{\nu}_i i \mid \tilde{\nu} \textrm{ has finite support} , \tilde{\nu}_i \in \mathbb{Z}_{\geq 0} \hspace{0.75em} \forall i \}$. Elements $\tilde{\nu} \in \mathbb{N}\tilde{I}$ are called \textbf{dimension vectors}. For $\tilde{\nu} \in \mathbb{N}\tilde{I}$, the \textbf{height} of $\tilde{\nu}$ is defined to be
\begin{equation*}
|\tilde{\nu}|=\sum_{i \in \tilde{I}}\tilde{\nu}_i.
\end{equation*}
For $\tilde{\nu} \in \mathbb{N}\tilde{I}$ with $|\tilde{\nu}|=m$, define
\begin{equation*}
\tilde{I}^{\tilde{\nu}}:=\{ \textbf{i}=(i_1, \ldots, i_m) \in \tilde{I}^m \mid \sum_{k=1}^m i_k = \tilde{\nu} \}.
\end{equation*}
\begin{defn}
$\tilde{\nu} = \sum_{i \in \tilde{I}} \tilde{\nu}_i i \in \mathbb{N}\tilde{I}$ is said to have \textbf{multiplicity one} if $\tilde{\nu}_i \leq 1$ for every $i \in \tilde{I}$. We say that $j \in \tilde{I}$ has multiplicity one in $\tilde{\nu}$, or $j$ appears with multiplicity one in $\tilde{\nu}$, if the coefficient of $j$ in $\tilde{\nu}$ is 1, i.e. if $\tilde{\nu}_j=1$.
\end{defn}
\begin{example}
$\tilde{\nu}_1=\lambda+p^2\lambda \in \mathbb{N}\tilde{I}_\lambda$ has multiplicity one, while $\tilde{\nu}_2=2\lambda+p^2\lambda$ does not have multiplicity one. In the latter example, $p^2\lambda$ appears with multiplicity one in $\tilde{\nu}_2$.
\end{example}
\begin{defn} \label{definition:klr algebra and relations}
For $\tilde{\nu}\in \mathbb{N}\tilde{I}$ with $|\tilde{\nu}|=m$ the \textbf{KLR algebra}, denoted by $\textbf{R}_{\tilde{\nu}}$, is the graded \textbf{k}-algebra generated by elements 
\begin{equation*}
\{ x_1, \ldots, x_m \} \cup \{ \sigma_1, \ldots, \sigma_{m-1} \} \cup \{ \idemp \mid \textbf{i} \in \tilde{I}^{\tilde{\nu}} \}
\end{equation*}
which are subject to the following relations. 
\begin{enumerate}
\item $\textbf{e}(\textbf{i})\textbf{e}(\textbf{j}) = \delta_{\textbf{ij}}\textbf{e}(\textbf{i})$, \hspace{4 pt} $\sigma_{k}\textbf{e}(\textbf{i}) = \textbf{e}(s_{k}\textbf{i})\sigma_{k}$, \hspace{4 pt} $x_{l}\textbf{e}(\textbf{i})=\textbf{e}(\textbf{i})x_{l}$, \hspace{4 pt} $\sum_{\textbf{i} \in \tilde{I}^{\tilde{\nu}}} \idemp = 1$.
\item The $x_l$'s commute.
\vspace*{0.5em}
\item $\sigma_{k}^{2}\textbf{e}(\textbf{i}) = \left\{
\begin{array}{l l l l}
    \textbf{e}(\textbf{i}) & \quad i_{k} \nleftrightarrow i_{k+1} \\
    (x_{k+1} - x_{k})\textbf{e}(\textbf{i}) & \quad i_{k} \leftarrow i_{k+1}\\
    (x_{k} - x_{k+1})\textbf{e}(\textbf{i}) & \quad i_{k} \rightarrow i_{k+1} \\
(x_{k+1}-x_k)(x_k-x_{k+1})\idemp & \quad i_k \leftrightarrow i_{k+1} \\    
    0 & \quad i_k=i_{k+1}
\end{array} \right. $\\ \\ \vspace*{0.5em}
$\sigma_{j}\sigma_{k}=\sigma_{k}\sigma_{j} \quad $ for $j \neq k\pm 1$ \\ \vspace*{0.5em}
$(\sigma_{k+1}\sigma_{k}\sigma_{k+1} - \sigma_{k}\sigma_{k+1}\sigma_{k})\idemp = \left\{
\begin{array}{l l l l}
    0 & \quad i_k \neq i_{k+2} \textrm{ or } i_k \nleftrightarrow i_{k+1} \\
    \idemp & \quad i_k=i_{k+2} \textrm{ and } i_{k} \rightarrow i_{k+1}\\
    -\idemp & \quad i_k=i_{k+2} \textrm{ and } i_{k} \leftarrow i_{k+1} \\
(2x_{k+1}-x_{k+2}-x_k)\idemp & \quad i_k=i_{k+2} \textrm{ and } i_k \leftrightarrow i_{k+1}.
\end{array} \right.$ \vspace*{0.5em}
\item $(\sigma_k x_l - x_{s_k(l)}\sigma_k)\idemp=\left\{
\begin{array}{l l l}
			-\idemp & \quad \textrm{ if } l=k, i_k=i_{k+1}\\
			\idemp & \quad \textrm{ if } l=k+1, i_k=i_{k+1}\\
			0 & \quad \textrm{ else.}
\end{array} \right. $\\
\end{enumerate}
\end{defn}
The grading on $\textbf{R}_{\tilde{\nu}}$ is given as follows.
\begin{equation*}
\begin{split}
&\textrm{deg}(\idemp)=0 \\
&\textrm{deg}(x_l \idemp)=2 \\
&\textrm{deg}(\sigma_k \idemp)=\left\{
\begin{array}{l l l}	
|i_k \rightarrow i_{k+1}|+|i_{k+1} \rightarrow i_k| & \quad \textrm{ if  } i_k \neq i_{k+1} \\
-2 & \quad \textrm{ if  } 	i_k = i_{k+1}	
\end{array} \right. 
\end{split}
\end{equation*}
where $|i_{k} \rightarrow i_{k+1}|$ denotes the number of arrows from $i_{k}$ to $i_{k+1}$ in the quiver $\tilde{\Gamma}$.
\\
\\
If $\tilde{\nu}=0$ we set $\textbf{R}_{\tilde{\nu}}=\textbf{k}$ as a graded $\textbf{k}$-algebra.
\begin{remark}
In this paper the underlying quiver $\tilde{\Gamma}$ for KLR algebras is always of type $A$ however, KLR algebras can be defined more generally.
\end{remark}
At this point we introduce some notation. Take any $\tilde{\nu} \in \mathbb{N}\tilde{I}$, of height $|\tilde{\nu}|=n$, and the associated KLR algebra $\subklr$. Whenever we work with the KLR algebra $\subklr$, for each $w \in \sn$, we must choose and fix a reduced expression of $w$, say $w=s_{i_1} \cdots s_{i_r}$ where $1 \leq i_k < n$, for all $k$. Then define $\sigma_{\dot{w}} \in \subklr$ as follows.
\begin{equation*}
\begin{split}
&\sigma_{\dot{w}}\idemp= \sigma_{i_1}\sigma_{i_2} \cdots \sigma_{i_r}\idemp \\
&\sigma_{\dot{1}}\idemp=\idemp
\end{split}
\end{equation*}
where $1$ denotes the identity element in $\sn$. Note that reduced expressions of $w$ are not always unique and so $\sigma_{\dot{w}}$ depends upon the choice of reduced expression of $w$. For example, $s_is_{i+1}s_i=s_{i+1}s_is_{i+1}$ in $\sn$, whereas the defining relations for the KLR algebra state that we do not always have $\sigma_i\sigma_{i+1}\sigma_i\idemp = \sigma_{i+1}\sigma_i\sigma_{i+1}\idemp$. In particular, this means that for elements $w_1,w_2 \in \sn$, $w_1=w_2$ does not necessarily imply $\sigma_{\dot{w_1}}\idemp=\sigma_{\dot{w_2}}\idemp$.
\begin{note}
If we have fixed a reduced expression for $w$ we will usually omit the dot and write $\sigma_w$ for $\sigma_{\dot{w}}$.
\end{note}
\begin{lemma}[Basis Theorem for KLR Algebras][\cite{DiagrammaticApproach}, Theorem 2.5] \label{lemma:KLR basis theorem}
Take $\tilde{\nu} \in \mathbb{N}\tilde{I}$ with $|\tilde{\nu}|=m$. The elements 
\begin{equation*}
\lbrace \sigma_{\dot{w}} x_1^{n_1}\cdots x_m^{n_m}\idemp \mid w \in \mathfrak{S}_m, \textbf{i} \in I^{\tilde{\nu}}, n_i \in \mathbb{N}_0 \textrm{ } \forall i\rbrace
\end{equation*}
form a \textbf{k}-basis for $\textbf{R}_{\tilde{\nu}}$.
\end{lemma}
\addtocontents{toc}{\protect\setcounter{tocdepth}{1}}
\subsection{Root Partitions Associated to $\subklr$} \label{subsection:klr root partitions}
\addtocontents{toc}{\protect\setcounter{tocdepth}{2}}
To each $i\in \tilde{I}$ we associate an integer $n$; the power of $p$ at that vertex. For example, to $p^2\lambda$ we associate 2. To $p^{-4}\lambda$ we associate $-4$. If $p\in \tilde{I}$; to the vertex $p^{2k+1}$, $k \in \mathbb{Z}$ we associate $2k+1$. So $i \in \tilde{I}$ is identified with an integer and refers to the power of $p$ at that vertex. The order on $\tilde{I}$ is the natural order on $\mathbb{Z}$.
\\
\\
Associated to the quiver $\Gamma=\Gamma_{\tilde{I}}$, with vertex set $\tilde{I}$, is a Cartan matrix $(a_{ij})_{i,j \in \tilde{I}}$ defined by
\begin{equation*}
a_{i,j}:= \left\{
\begin{array}{l l l l}
2 & \quad \textrm{if } i=j \\
0 & \quad \textrm{if } i\nleftrightarrow j\\
-1 & \quad \textrm{if } i\rightarrow j \textrm{ or } i\leftarrow j\\
-2 & \quad \textrm{if } i\leftrightarrows j.
\end{array} \right. 
\end{equation*}
Let $(\mathfrak{h},\Pi,\Pi^\vee)$ be a realisation of this Cartan matrix. We have simple roots in type $A$ given by $\lbrace \alpha_i \mid i \in \tilde{I} \rbrace$, and positive roots $\alpha_{i,i+2k}:=\alpha_i+\alpha_{i+2}+\cdots + \alpha_{i+2k}$, $k \in \mathbb{Z}_{\geq 0}$. We order the roots as follows. 
\begin{equation*}
\alpha_i+\alpha_{i+2}+\cdots + \alpha_{i+2k} > \alpha_j+\alpha_{j+2}+\cdots + \alpha_{j+2l} \iff i>j \textrm{ or } i=j \textrm{ and } k>l.
\end{equation*}
\begin{defn}
A \textbf{root partition} of $\tilde{\nu}$ is a tuple of positive roots $(\beta_1, \beta_2, \beta_3, \ldots, \beta_r)$ such that $\tilde{\nu}
= \beta_1 + \beta_2 + \beta_3 + \cdots + \beta_r$ and such that $\beta_1 \geq \beta_2 \geq \beta_3 \geq \cdots \geq \beta_r$.
\end{defn}
A root partition of $\tilde{\nu}$ will often be denoted by $\pi$ and the set of all root partitions of $\tilde{\nu}$ will be denoted by $\Pi(\tilde{\nu})$.
\begin{example} 
Take $\tilde{\nu}=\lambda+p^2\lambda+p^4\lambda+p^6\lambda \in \mathbb{N}\tilde{I}_\lambda$. $(\lambda+p^2\lambda+p^4\lambda+p^6\lambda)$, $(p^6\lambda,\lambda+p^2\lambda+p^4\lambda)$, $(p^6\lambda,p^2\lambda+p^4\lambda, \lambda)$, are root partitions of $\tilde{\nu}$. Denote them as $\pi_1$, $\pi_2$, $\pi_3$, respectively. We have
\begin{equation*}
\pi_1 < \pi_2 < \pi_3
\end{equation*}
and $\textbf{i}_{\pi_1}=(\lambda, p^2\lambda, p^4\lambda, p^6\lambda)$, $\textbf{i}_{\pi_2}=(p^6\lambda,\lambda, p^2\lambda, p^4\lambda)$, $\textbf{i}_{\pi_3}=(p^6\lambda,p^2\lambda, p^4\lambda,\lambda)$.
\end{example}
Given $\tilde{\nu} \in \mathbb{N}\tilde{I}$ and the associated KLR algebra $\subklr$, we have a set $\{ \idemp \mid \textbf{i} \in \tilde{I}^{\tilde{\nu}} \}$ of idempotents. Each $\idemp$ is labelled by a sequence of integers $\textbf{i}=(i_1, \ldots, i_m) \in \tilde{I}^m$ which may or may not correspond to a root partition in the way we describe above. So, associated to each KLR algebra $\subklr$ we obtain a set of root partitions of $\tilde{\nu}$. In fact, since all permutations of $(i_1, \ldots, i_m)$ lie in $\tilde{I}^m$, we obtain a complete set of root partitions of $\tilde{\nu}$. If $\textbf{i}=(i_1, \ldots, i_m)$ corresponds to a root partition $\pi \in \Pi(\tilde{\nu})$ then we will sometimes write $\textbf{e}(\textbf{i}_\pi)=\idemp$ to emphasise this.
\subsection{VV Algebras} \label{section:vv definition}
In this section we define the family of algebras of which this paper is concerned; the VV algebras. They were introduced by Varagnolo and Vasserot in \cite{VaragnoloVasserot}.
\\
\\
Fix elements $p, q \in \mathbf{k}^{\times}$. Assume that $p$ is not a power of $q$ and that $q$ is not a power of $p$. Define an action of $\mathbb{Z} \rtimes \mathbb{Z}_{2}$ on $\mathbf{k}^{\times}$ as follows.
\begin{equation*}
(n,\varepsilon) \cdot \lambda = p^{2n}\lambda^{\varepsilon}.
\end{equation*}
Let $I$ be a $\mathbb{Z} \rtimes \mathbb{Z}_2$-orbit. So $I=I_\lambda$ is the $\mathbb{Z} \rtimes \mathbb{Z}_2$-orbit of $\lambda$.
\begin{equation*}
I=I_\lambda=\{ p^{2n}\lambda^{\pm 1} \mid n \in \mathbb{Z} \}.
\end{equation*}
To $I$ we associate a quiver $\Gamma=\Gamma_I$ together with an involution $\theta$. The vertices of $\Gamma$ are the elements $i\in I$ and we have arrows $p^{2}i \longrightarrow i$ for every $i\in I$. The involution $\theta$ is defined by 
\begin{equation*}
\begin{split}
\theta(i)&=i^{-1} \\
\theta(p^{2}i \longrightarrow i)&=p^{-2}i^{-1} \longleftarrow i^{-1}, \quad \textrm{ for all } i \in I .
\end{split}
\end{equation*}
We always assume that $\pm 1 \notin I$ and that $p \neq \pm 1$. This implies that $\theta$ has no fixed points and that $\Gamma$ has no loops (1-cycles). 
\\
\\
Now define ${^\theta}\mathbb{N}I:=\{ \nu = \sum_{i\in I}\nu_i i \mid \nu \textrm{ has finite support} , \nu_i \in \mathbb{Z}_{\geq 0}, \nu_i=\nu_{\theta(i)} \hspace{0.75em} \forall i \}$. In particular, for each $\nu \in {^\theta}\mathbb{N}I$, the coefficients of $i$ and $i^{-1}$ in $\nu$ must be equal. Elements $\nu \in {^\theta}\mathbb{N}I$ are called \textbf{dimension vectors}. For $\nu \in {^\theta}\mathbb{N}I$ the \textbf{height} of $\nu$ is defined to be,
\begin{equation*}
|\nu|=\sum_{i \in I}\nu_i.
\end{equation*}
The shape of $\Gamma$ depends on whether $p \in I$ or $p \not\in I$, as well as whether or not $p$ is a root of unity.
\begin{itemize}
\item Suppose $p \not\in I$. Let $I^+_\lambda:=\{ p^{2n}\lambda \mid n \in \mathbb{Z} \}$, $I^-_\lambda:=\{ p^{2n}\lambda^{-1} \mid n \in \mathbb{Z} \}$. So $I_\lambda=I^-_\lambda \sqcup I^+_\lambda$. 
\item Now suppose $p \in I$. Let $I^+_p:=\{ p^{2n+1} \mid n \in \mathbb{Z}_{\geq 0} \}$, $I^-_p:=\{ p^{2n-1} \mid n \in \mathbb{Z}_{\leq 0} \}$. So $I_p=I^-_p \sqcup I^+_p$, provided $p$ is not a root of unity in which case we have $I_p^-=I_p^+$.
\end{itemize}
Note that for $\nu \in {^\theta}\mathbb{N}I$,  $|\nu|=2m$ for some positive integer $m$. For $\nu \in {^\theta}\mathbb{N}I$ with $|\nu|=2m$, define
\begin{equation*}
{^\theta}I^\nu:=\{ \textbf{i}=(i_1, \ldots, i_m) \in I^m \mid \sum_{k=1}^m i_k + \sum_{k=1}^m i_k^{-1} = \nu \}.
\end{equation*}
\begin{defn}
A dimension vector $\nu = \sum_{i \in I} \nu_i i \in {^\theta}\mathbb{N}I$ is said to have \textbf{multiplicity one} if $\nu_i \leq 1$ for every $i \in I$. We say that $j \in I$ has multiplicity one in $\nu$, or $j$ appears with multiplicity one in $\nu$, if the coefficient of $j$ in $\nu$ is 1, i.e. if $\nu_j=1$.
\end{defn}
\begin{remark}
We remark here that, given this data, we can again define KLR algebras. That is, given a $\mathbb{Z} \rtimes \mathbb{Z}_2$-orbit $I_\lambda$ and associated quiver $\Gamma_{I_\lambda}$ we can pick $\tilde{\nu} \in \mathbb{N}I_\lambda$ which yields a KLR algebra $\subklr$. This KLR algebra is defined by the generators and relations given in \ref{klr} with the understanding that there are no arrows between vertices belonging to different branches of $\Gamma_{I_\lambda}$.
\end{remark}
\begin{defn} \label{definition:vv algebra}
For $\nu\in {^\theta}\mathbb{N}I$ with $|\nu|=2m$ the \textbf{VV algebra}, denoted by $\vv$, is the graded \textbf{k}-algebra generated by elements 
\begin{equation*}
\{ x_1, \ldots, x_m \} \cup \{ \sigma_1, \ldots, \sigma_{m-1} \} \cup \{ \idemp \mid \textbf{i} \in {^\theta}I^\nu \} \cup \{ \pi \}
\end{equation*}
which are subject to the relations defining the KLR algebras given in \ref{definition:klr algebra and relations} together with,
\begin{enumerate}
\item $\pi \textbf{e}(i_1,\ldots,i_m) = \textbf{e}(i_1^{-1},i_2 \ldots,i_m) \pi.$ \vspace*{0.5em}
\item $\pi^{2}\textbf{e}(\textbf{i})=\left\{
\begin{array}{l l l}
    x_1\idemp & \quad i_1=q \\
    -x_1\idemp & \quad i_1=q^{-1} \\
    \idemp & \quad i_1 \neq q^{\pm 1}.
\end{array} \right.$ \vspace*{0.5em}
\item $\pi x_{1} = -x_{1} \pi$ \\
$\pi x_{l} = x_{l} \pi$ \textrm{ for all } $l > 1.$
\item $(\sigma_1 \pi)^2\idemp - (\pi\sigma_1)^2\idemp= \left\{
\begin{array}{l l l}
0 & \quad i_1^{-1}\neq i_2 \textrm{ or if } i_1 \neq q^{\pm 1} \\
\sigma_1\idemp & \quad i_1^{-1}=i_2=q^{-1} \\
-\sigma_1\idemp & \quad i_1^{-1}=i_2=q
\end{array} \right. $ \\
$\pi \sigma_{k}=\sigma_{k} \pi \quad \forall k\neq 1.$
\end{enumerate}
\end{defn}
The grading on $\mathfrak{W}_{\nu}$ is defined as follows.
\begin{equation*}
\begin{split}
&\textrm{deg}(\idemp)=0 \\
&\textrm{deg}(x_l \idemp)=2 \\
&\textrm{deg}(\pi\idemp)=\left\{
\begin{array}{l l}	
1 & \quad \textrm{ if  } i_1=q^{\pm 1} \\
0 & \quad \textrm{ if  } i_1\neq q^{\pm 1}	
\end{array} \right. \\
&\textrm{deg}(\sigma_k \idemp)=\left\{
\begin{array}{l l l}	
|i_k \rightarrow i_{k+1}|+|i_{k+1} \rightarrow i_k| & \quad \textrm{ if  } i_k \neq i_{k+1} \\
-2 & \quad \textrm{ if  } 	i_k = i_{k+1}.	
\end{array} \right. 
\end{split}
\end{equation*}
where $|i_{k} \rightarrow i_{k+1}|$ denotes the number of arrows from $i_{k}$ to $i_{k+1}$ in the quiver $\Gamma$. 
\\
\\
If $\nu=0$ we set $\vv=\textbf{k}$ as a graded $\textbf{k}$-algebra.
\begin{remark} \label{definition:KLR^+ and KLR^-}
Every $\nu \in {^\theta}\mathbb{N}I$ can be written as
\begin{equation*}
\nu = \sum_{i \in I^+}\nu_ii + \sum_{i \in I^-}\nu_ii.
\end{equation*}
Setting $\tilde{\nu}= \sum_{i \in I^+}\nu_ii \in \mathbb{N}I^+$ defines a KLR algebra. Denote the KLR algebra associated to $\tilde{\nu} \in \mathbb{N}I^+$ by $\subklr^+$. For the remainder of this paper, for $\nu = \sum_{i \in I^+}\nu_ii + \sum_{i \in I^-}\nu_ii \in {^\theta}\mathbb{N}I$, $\tilde{\nu}$ will be used to denote $\sum_{i \in I^+}\nu_ii \in \mathbb{N}I^+$. Sometimes we write $\tilde{\nu}=\tilde{\nu}^+$ to make this explicit.
\\
\\
Similarly, setting $\tilde{\nu}^-=\sum_{i \in I^-}\nu_ii \in \mathbb{N}I^-$ also defines a KLR algebra. Denote the KLR algebra associated to $\tilde{\nu}^- \in \mathbb{N}I^-$ by $\subklr^-$. 
\end{remark}
The relations above, and therefore the algebras $\vv$, depend on the following four cases.
\begin{enumerate}
\item \textbf{The case $\textbf{p, q} \boldsymbol\not\boldsymbol\in \textbf{I}$}. In this setting $\Gamma_{I_\lambda}$ is of type $A^\infty_\infty \sqcup A^\infty_\infty$ if $p$ is not a root of unity. If $p^2$ is an $r^\textrm{th}$ primitive root of unity then $\Gamma_{I_\lambda}$ is of type $A^{(1)}_r \sqcup A^{(1)}_r$. Note that we do not allow $p^{2n+1}=1$ since $\pm 1 \not\in I$.
\item \textbf{The case $\textbf{p} \boldsymbol\in \textbf{I}$ and $\textbf{q} \boldsymbol\not\boldsymbol\in \textbf{I}$}. If $p\in I$ and $p$ is not a root of unity then $\Gamma_{I_p}$ has type $A^\infty_\infty$ and $\Gamma_{I_p}$ is of type $A^{(1)}_{r}$ if $p^2$ is an $\textrm{r}^{\textrm{th}}$ root of unity.
\item \textbf{The case $\textbf{q} \boldsymbol\in \textbf{I}$ and $\textbf{p} \boldsymbol\not\boldsymbol\in \textbf{I}$}. In this setting $\Gamma_{I_q}$ is of type $A^\infty_\infty \sqcup A^\infty_\infty$ if $p$ is not a root of unity. If $p$ is a  root of unity then $\Gamma_{I_q}$ is of type $A^{(1)}_{r}\sqcup A^{(1)}_{r}$.
\item \textbf{The case $\textbf{p}, \textbf{q} \boldsymbol\in \textbf{I}$}. In this case $I_p=I_q$. That is, 
\begin{equation*}
\{ p^{2n+1} \mid n \in \mathbb{Z} \}=\{ p^{2n}q^{\pm 1} \mid n \in \mathbb{Z} \}.
\end{equation*}
If $q=p^{2n}$, for some $n \in \mathbb{Z}$, then $\pm 1 \in I$, which we have ruled out. So we must have $q=p^{2n+1}$ for some $n \in \mathbb{Z}$.
\end{enumerate}
Here we introduce some notation similar to the notation used for KLR algebras. Take any $\nu \in {^\theta}\mathbb{N}I$, of height $|\nu|=2m$, and the associated VV algebra $\vv$. Whenever we work with the VV algebra $\vv$, for each $w \in W^B_m$, we must choose and fix a reduced expression, say $w=s_{i_1} \cdots s_{i_r}$ where $0 \leq i_k < m$, for all $k$. Then define $\sigma_{\dot{w}} \in \vv$ as follows.
\begin{equation*}
\begin{split}
&\sigma_{\dot{w}}\idemp= \sigma_{i_1}\sigma_{i_2} \cdots \sigma_{i_r}\idemp \\
&\sigma_{\dot{1}}\idemp=\idemp
\end{split}
\end{equation*}
where $1$ is the identity element in $W^B_m$. As with elements of the symmetric group, reduced expressions of $w \in W^B_m$ are not always unique and so $\sigma_{\dot{w}}$ depends upon the choice of reduced expression of $w$. 
\begin{note}
If we have fixed a reduced expression for $w \in W^B_m$ we will usually omit the dot and write $\sigma_w$ instead of $\sigma_{\dot{w}}$.
\end{note}
For each $\nu \in {^\theta}\mathbb{N}I$, with $|\nu|=2m$, define $^{\theta}\textbf{F}_{\nu}$ to be a polynomial ring in the $x_k$ at each $\idemp$. More precisely,
\begin{displaymath}
{^\theta}\textbf{F}_{\nu} := \bigoplus_{\textbf{i} \in {^\theta}I^{\nu}} \textbf{k}[x_{1}\idemp, \ldots, x_{m}\idemp].
\end{displaymath}
\begin{proposition}[\cite{VaragnoloVasserot}, Proposition 7.5]
The \textbf{k}-algebra $\vv$ is a free (left or right) ${^\theta}\textbf{F}_{\nu}$-module on basis $\{ \sigma_{\dot{w}} \mid w \in W^B_n \}$. It has rank $2^mm!$. The operator $\sigma_{\dot{w}}\idemp$ is homogeneous and its degree is independent of the choice of reduced expression of $\dot{w}$.
\end{proposition}
That is, 
\begin{displaymath}
\vv = \bigoplus_{\lbrace \sigma_{\dot{w}} \mid w\in W^B_m \rbrace} {^{\theta}\mathbf{F}_{\nu}}.
\end{displaymath}
Then we have a $\textbf{k}$-basis for VV algebras, as follows.
\begin{lemma}[Basis Theorem for VV Algebras] \label{vv basis}
Take $\nu \in \mathbb{N}I$ with $|\nu|=2m$. The elements 
\begin{equation*}
\lbrace \sigma_{\dot{w}} x_1^{n_1}\cdots x_m^{n_m}\idemp \mid w\in W^B_m, \textbf{i} \in {^\theta}I^\nu, n_k \in \mathbb{N}_0 \textrm{ } \forall k\rbrace
\end{equation*}
form a \textbf{k}-basis for $\vv$.
\end{lemma}
\addtocontents{toc}{\protect\setcounter{tocdepth}{1}}
\subsection{KLR Algebras as Idempotent Subalgebras}
\addtocontents{toc}{\protect\setcounter{tocdepth}{2}}
The following proposition provides more of an understanding of the relationship between KLR algebras and VV algebras.
\begin{proposition} \label{proposition:klr an idempotent subalgebra} For any $\nu \in {^\theta}\mathbb{N}I$ we can write $\nu=\sum_{i \in I} \nu_i i + \sum_{i \in I} \nu_i i^{-1}$ and set $\tilde{\nu}=\sum_{i \in I} \nu_i i$. If $i+i^{-1}$ is not a summand of $\tilde{\nu}$, for any $i\in I$, then $\tilde{\nu}$ yields an idempotent subalgebra $\textbf{R}_{\tilde{\nu}}$ of $\vv$.
\end{proposition}
\begin{proof}
We will show that there is an algebra isomorphism
\begin{equation*}
\textbf{e}\vv\textbf{e} \cong \subklr \textrm{ where } \textbf{e}=\sum_{\textbf{i} \in I^{\tilde{\nu}}} \idemp. 
\end{equation*} 
Let $\mathcal{D}_m:=\mathcal{D}(W^B_m/\sm)$ denote the minimal length left coset representatives of $\sm$ in $W^B_m$. Fix a reduced expression $\dot{s}$ for every $s\in \sm$. It is well-known (see \cite{GeckPfeiffer} Proposition (2.1.1), for example) that every $w \in W^B_m$ can be written uniquely in the form $\eta s$, for $\eta \in \dee$, $s \in \sm$, with $\ell(\eta s)=\ell(\eta)+\ell(s)$. For every $w \in W^B_m$ fix a reduced expression $\dot{w}=\dot{\eta}\dot{s}$. By the basis theorem for VV algebras \ref{vv basis}, any element $v \in \vv$ can be expressed in the following form. 
\begin{equation*}
v=\sum_{\substack {w \in W^B_m \\ \textbf{i} \in {^\theta}I^\nu}} \sigma_w p_\textbf{i}(\underline{x})\idemp=\sum_{\substack{s_{i_k} \in \sm \\ \eta \in \mathcal{D}_m \\ \textbf{i} \in {^\theta}I^\nu}} \sigma_\eta \sigma_{s_{i_1}\cdots s_{i_r}}p_\textbf{i}(\underline{x})\idemp
\end{equation*}
where $\underline{x}=(x_1, \ldots, x_m)$, and $p_\textbf{i}(\underline{x}) \in \textbf{k}[x_1\idemp, \ldots, x_m\idemp]$. Then 
\begin{equation*}
\textbf{e} v \textbf{e}=\sum_{\substack{s_{i_k}\in \sm \\ \eta \in \mathcal{D}_m \\ \textbf{i} \in I^{\tilde{\nu}}}} \textbf{e}\sigma_\eta \textbf{e}\sigma_{s_{i_1}\cdots s_{i_r}} p_\textbf{i}(\underline{x}).
\end{equation*}
\begin{claim}
For $\eta \in \mathcal{D}(W^B_m/\sm)$,
\begin{equation*}
\textbf{e}\sigma_\eta \textbf{e}=\left\{
\begin{array}{l l}
\textbf{e} & \textrm{if } \eta=1 \\
0 & \textrm{else.}
\end{array} \right.
\end{equation*}
\end{claim}
This is clear when $\eta = 1$. So suppose that $\eta \neq 1$. We prove by induction on $m$ that $\textbf{e}\sigma_\eta \textbf{e}=0$. \\
\\
For $m=1$ we have $\mathcal{D}_1 = \{ 1, s_0 \}$ and $\tilde{\nu}=a \in \mathbb{N}I$. Clearly $\textbf{e}\sigma_0\textbf{e}=\textbf{e}\pi\textbf{e}=0$. Now for any $\tilde{\nu} \in \mathbb{N}I$ of height $k<m$ assume  $\textbf{e}\sigma_\eta \textbf{e}=0$ for all $\eta \in \mathcal{D}_k$. 
\\
\\
Take $\tilde{\nu} \in \mathbb{N}I$ of height $m$ and $\eta \in \mathcal{D}_m \setminus \mathcal{D}_{m-1}$. Then, $\eta=\bar{\eta}s_{m-1}\cdots s_1s_0$, for $\bar{\eta} \in \mathcal{D}_{m-1}$. Consider one summand of $\textbf{e}$, say $e(a_1, \ldots, a_m)$. Then,
\begin{equation*}
\begin{split}
\textbf{e} \sigma_\eta e(a_1, \ldots, a_m)&=\textbf{e}\sigma_{\bar{\eta}}\sigma_{m-1}\cdots \sigma_1\pi e(a_1, \ldots, a_m) \\
&=\textbf{e}\sigma_{\bar{\eta}}e(a_2,\ldots,a_m,a_1^{-1})\sigma_{m-1}\cdots \sigma_1\pi.
\end{split}
\end{equation*}
Since $a_1$ is a summand of $\tilde{\nu}$, by assumption, $a_1^{-1}$ is not an entry of any $(i_1,\ldots, i_m) \in I^{\tilde{\nu}}$ (otherwise $a_1+a_1^{-1}$ would be a summand of $\tilde{\nu}$). Also, $\sigma_{\bar{\eta}}$ does not affect the entry $a_1^{-1}$, or its position, in $e(a_1, \ldots, a_m, a_1^{-1})$. Hence we have 
\begin{equation*}
\textbf{e}\sigma_{\bar{\eta}} e(a_2, \ldots, a_m, a_1^{-1})=0 \textrm{, and so } \textbf{e}\sigma_\eta e(a_1, \ldots, a_m)=0.
\end{equation*}
This is true for every summand of $\textbf{e}$ and so $\textbf{e}\sigma_{\eta}\textbf{e}=0$, for $\eta \neq 1$. 
\\
\\
Then,
\begin{equation*}
\textbf{e} v \textbf{e}= 
\left\lbrace
\begin{array}{ll}
\sum_{\substack{s \in \sm \\ \textbf{i} \in I^{\tilde{\nu}}}} \sigma_s p_\textbf{i}(\underline{x})\textbf{e}  & \textrm{ if } \eta =1 \\
0 & \textrm{ if } \eta \neq 1.
\end{array}
\right.
\end{equation*}
Now define a map
\begin{equation*}
\begin{split}
f: \textbf{e}\vv\textbf{e} &\longrightarrow \subklr \\
\textbf{e}v\textbf{e} &\mapsto \sum_{\substack{s \in \sm \\ \textbf{i} \in I^{\tilde{\nu}}}} \sigma_s p_\textbf{i}(\underline{x})\textbf{e}.
\end{split}
\end{equation*}
From \ref{lemma:KLR basis theorem} we know $\{ \sigma_s x_1^{n_1}\cdots x_m^{n_m} \idemp \mid s \in \sm, \textbf{i} \in I^{\tilde{\nu}}, n_k \in \mathbb{N}_0 \hspace{0.2em} \forall k \}$ is a basis for $\subklr$ and so $f$ is an isomorphism of $\textbf{k}$-vector spaces. On inspection of the defining relations of the KLR algebras and VV algebras it follows that $f$ is in fact a morphism of algebras. This completes the proof.
\end{proof}
\begin{remark} \label{remark:every VV has KLR+ as idemp subalg}
Using the notation from Remark \ref{definition:KLR^+ and KLR^-}, every VV algebra $\vv$ has idempotent subalgebras $\subklr^-$ and $\subklr^+$.
\end{remark}
\begin{example}
Take $\nu=p^{-2}\lambda^{-1} + \lambda^{-1} + \lambda + p^2\lambda \in {^\theta}\mathbb{N}I$. The algebra $\subklr^+$, corresponding to $\tilde{\nu}=\lambda+p^2\lambda$, is always an idempotent subalgebra of $\vv$, as is $\subklr^-$ which is the KLR algebra corresponding to $\tilde{\nu}= \lambda^{-1} +p^{-2}\lambda^{-1}$. But $\subklr$ is also an idempotent subalgebra for $\tilde{\nu} =\lambda^{-1}+p^2\lambda$ and for $\tilde{\nu} =\lambda+p^{-2}\lambda^{-1}$. In this example $\vv$ has four KLR algebras appearing as idempotent subalgebras.
\end{example}
\begin{example} 
Take $\nu = 2p^{-1}+2p \in {^\theta}\mathbb{N}I_p$. The algebra $\subklr$, for $\tilde{\nu}=p+p^{-1}$, is not an idempotent subalgebra of $\vv$. In this example, the only idempotent subalgebras isomorphic to KLR algebras are $\subklr^+$ and $\subklr^-$.
\end{example}
\section{Morita Equivalences Between KLR and VV Algebras} \label{chapter:MEs between KLR and VV}
\subsection{Morita Equivalence in the Separated Case} \label{section: vvlm ME to vv1 x vv2}
For a subset $I \subset \textbf{k}^\times$, let $H^B_m\textrm{-Mod}_I$ denote the category of modules over $H^B_m$ in which all eigenvalues of $X_i$, $1 \leq i \leq m$, lie in $I$. That is, $M \in H^B_m\textrm{-Mod}_I$ if and only if whenever $X_i n=\lambda n$ for some $n \in M$, $\lambda \in \textbf{k}$ and for some $1 \leq i \leq m$, then $\lambda \in I$. In this case we say that $M$ is of type $I$. We start this section by referencing the following lemma.
\begin{lemma}[\cite{EnomotoKashiwara}, Lemma 3.5] \leavevmode
\begin{enumerate}
\item Let $N'$ be a simple $\mathrm{H}^B_{n'}$-module of type $I$ and $N''$ a simple $\mathrm{H}^B_{n''}$-module of type J. Then $N'\boxtimes N''$ is a simple $\mathrm{H}^B_{n'} \otimes \mathrm{H}^B_{n''}$-module and $\mathrm{Ind}^{\mathrm{H}^B_{n'+n''}}_{\mathrm{H}^B_{n'} \otimes \mathrm{H}^B_{n''}}N'\boxtimes N''$ is a simple $\mathrm{H}^B_{n'+n''}$-module of type $I\cup J$.
\item Conversely if $M$ is a simple $\mathrm{H}^B_{m}$-module of type $I\cup J$ then there exists a simple $\mathrm{H}^B_{n}$-module $N'$ of type $I$ and a simple $\mathrm{H}^B_{m-n}$-module $N''$ of type $J$ such that $M \cong \mathrm{Ind}^{\mathrm{H}^B_m}_{\mathrm{H}^B_n \otimes \mathrm{H}^B_{m-n}}N'\boxtimes N''$.
\end{enumerate}
\end{lemma}
In \cite{EnomotoKashiwara} they conclude that it suffices to study $\mathrm{H}^B_m$-modules of type $I$. In this section we provide more of a categorical justification of this fact.
\\
\\
Recall, from Section \ref{section:vv definition}, we fixed an element $p \in \textbf{k}^\times$ in order to define an action of $\mathbb{Z} \rtimes \mathbb{Z}_2$ on $\textbf{k}^\times$. We then fixed a $\mathbb{Z} \rtimes \mathbb{Z}_2$-invariant subset $I_\lambda$ of $\textbf{k}^\times$, for some $\lambda \in \textbf{k}^\times$, and defined a family of VV algebras from this data. In this section we will define families of VV algebras in a slightly more general setting and then show that in fact it suffices to study the families of VV algebras as defined in Section \ref{section:vv definition}. We again fix elements $p,q \in \textbf{k}^\times$ and keep the action of $\mathbb{Z} \rtimes \mathbb{Z}_2$ on $\textbf{k}^\times$ as before, namely,
\begin{equation*}
(n,\varepsilon) \cdot \lambda = p^{2n}\lambda^{\varepsilon}.
\end{equation*}
Let $I,J \subset \bk^\times$ be $\mathbb{Z}\rtimes \mathbb{Z}_2$-orbits with $I \cap J = \emptyset$. To $I \cup J$ we associate a quiver $\Gamma = \Gamma_{I \cup J}$ together with an involution $\theta$. The vertices of $\Gamma$ are the elements $i \in I \cup J$ and we have arrows $p^2i \longrightarrow i$ for every $i \in I \cup J$. The involution $\theta$ is defined as before;
\begin{equation*}
\begin{split}
\theta(i)&=i^{-1} \\
\theta(p^{2}i \longrightarrow i)&=p^{-2}i^{-1} \longleftarrow i^{-1} \quad \textrm{ for all } i \in I \cup J.
\end{split}
\end{equation*}
Define ${^\theta}\mathbb{N}(I \cup J):=\{ \nu = \sum_{i\in I \cup J}\nu_i i \mid \nu \textrm{ has finite support} , \nu_i \in \mathbb{Z}_{\geq 0}, \nu_i=\nu_{\theta(i)} \hspace{0.75em} \forall i \}$. For $\nu \in {^\theta}\mathbb{N}(I \cup J)$, the height of $\nu$ is defined to be
\begin{equation*}
|\nu|=\sum_{i \in I}\nu_i
\end{equation*}
and is equal to $2m$, for some positive integer $m$. For $\nu \in {^\theta}\mathbb{N}(I \cup J)$ of height $2m$, define
\begin{equation*}
{^\theta}(I \cup J)^\nu:=\{ \textbf{i}=(i_1, \ldots, i_m) \in (I \cup J)^m \mid \sum_{k=1}^m i_k + \sum_{k=1}^m i_k^{-1} = \nu \}.
\end{equation*} 
\begin{defn} \label{definition:separated vv algebra}
For $\nu\in {^\theta}\mathbb{N}(I \cup J)$ with $|\nu|=2m$ the \textbf{separated VV algebra}, denoted $\vv$, is the graded \textbf{k}-algebra generated by elements 
\begin{equation*}
\{ x_1, \ldots, x_m \} \cup \{ \sigma_1, \ldots, \sigma_{m-1} \} \cup \{ \idemp \mid \textbf{i} \in {^\theta}(I \cup J)^\nu \} \cup \{ \pi \}
\end{equation*}
which are subject to the relations given in Definition \ref{definition:vv algebra} with the understanding that there are no arrows between between vertices coming from $I$ and those coming from $J$.
\end{defn}
The grading on $\mathfrak{W}_{\nu}$ is defined as follows.
\begin{equation*}
\begin{split}
&\textrm{deg}(\idemp)=0 \\
&\textrm{deg}(x_l \idemp)=2 \\
&\textrm{deg}(\pi\idemp)=\left\{
\begin{array}{l l}	
1 & \quad \textrm{ if  } i_1=q^{\pm 1} \\
0 & \quad \textrm{ if  } i_1\neq q^{\pm 1}	
\end{array} \right. \\
&\textrm{deg}(\sigma_k \idemp)=\left\{
\begin{array}{l l l}	
|i_k \rightarrow i_{k+1}|+|i_{k+1} \rightarrow i_k| & \quad \textrm{ if  } i_k \neq i_{k+1} \\
-2 & \quad \textrm{ if  } 	i_k = i_{k+1}	
\end{array} \right. 
\end{split}
\end{equation*}
where $|i_{k} \rightarrow i_{k+1}|$ denotes the number of arrows from $i_{k}$ to $i_{k+1}$ in the quiver $\Gamma_{I \cup J}$. 
\\
\\
If $\nu=0$ we set $\vv=\textbf{k}$ as a graded $\textbf{k}$-algebra.
\begin{remark} \label{remark:sigma squares to 1 in separated VV}
Take $\nu \in {^\theta}\mathbb{N}(I \cup J)$ and the separated VV algebra $\vv$. We emphasise here that there are no arrows between any $i \in I$ and $j \in J$. So, for some $\idemp$ with $i_k=i$ and $i_{k+1}=j$, it is always the case that $\sigma_k^2 \idemp=\idemp$.
\end{remark}
\begin{note} \label{note:notation in the separated case}
Fix $\nu \in {^\theta}\mathbb{N}(I \cup J)$. We can write $\nu = \nu_1 + \nu_2$, for $\nu_1 \in {^\theta}\mathbb{N}I$ and $\nu_2 \in {^\theta}\mathbb{N}J$. We note here that $(i_1, \ldots, i_{m_1}, j_1, \ldots, j_{m_2}) \in {^\theta}(I \cup J)^\nu$ for any $(i_1, \ldots, i_{m_1}) \in {^\theta}I^{\nu_1}$ and any $(j_1, \ldots, j_{m_2}) \in {^\theta}J^{\nu_2}$. We will write $(\textbf{i}\textbf{j})$ for such a tuple, where $\textbf{i}=(i_1, \ldots, i_{m_1}) \in {^\theta}I^{\nu_1}$ and $\textbf{j}=(j_1, \ldots, j_{m_2}) \in {^\theta}J^{\nu_2}$. 
\end{note}
Let $\textbf{e}=\sum_{\substack{\textbf{i}\in {^\theta}I^{\nu_1} \\ \textbf{j}\in {^\theta}J^{\nu_2}}}\idempij$. In this section we will see that $\vv\textbf{e}$ is a progenerator in $\vv$-Mod such that $\textbf{e}\vv\textbf{e} \cong \w_{\nu_1} \otimes_{\textbf{k}} \w_{\nu_2}$, from which it follows that $\vv$ and $\w_{\nu_1} \otimes_{\textbf{k}} \w_{\nu_2}$ are Morita equivalent.
\\
\\
From here until the end of this section we fix the following notation. Let $I,J \subset \bk^\times$ be $\mathbb{Z}\rtimes \{ \pm 1 \}$-orbits with $I \cap J = \emptyset$. Take $\nu \in {^\theta}\mathbb{N}(I \cup J)$ and write $\nu = \nu_1 + \nu_2$, for $\nu_1 \in {^\theta}\mathbb{N}I$ with $|\nu_1|=m_1$ and $\nu_2 \in {^\theta}\mathbb{N}J$ with $|\nu_2|=m_2$. Let $m=m_1+m_2$. Set $\textbf{e}:= \sum_{\substack{\textbf{i} \in {^\theta}I^{\nu_1} \\ \textbf{j}\in {^\theta}J^{\nu_2}}}\idempij$, using the notation as in Note \ref{note:notation in the separated case}.
\\
\\
Let $Q:=\big\langle s_0, s_1, \ldots, s_{m_1-1}, s_{m_1}\cdots s_1s_0s_1\cdots s_{m_1}, s_{m_1+1}, \ldots, s_{m_1+m_2-1} \big\rangle \subset W^B_m$, a quasi-parabolic subgroup of $W^B_m$.
\begin{lemma} \label{lemma: group isomorphism}
There is a group isomorphism $W^B_{m_1} \times W^B_{m_2} \cong Q$.
\end{lemma}
\begin{proof}
Let
\begin{equation*}
\phi: W^B_{m_1} \times W^B_{m_2} \longrightarrow Q
\end{equation*}
be the group homomorphism defined on generators as follows.
\begin{equation*}
\begin{split}
(e_1, e_2) &\mapsto e \\
(s_i, e_2) &\mapsto s_i \quad 0\leq i \leq m_1-1 \\
(e_1, s_j) &\mapsto s_{j+m_1} \quad 1\leq j \leq m_2-1 \\
(e_1, s_0) &\mapsto s_{m_1}\cdots s_1s_0s_1 \cdots s_{m_1}.
\end{split}
\end{equation*}
One can check that $\phi$ really is a well-defined group homomorphism by checking the relations. In particular,
\begin{equation*}
\phi\big( (e_1,s_0s_1s_0s_1) \big)=\phi\big( (e_1,s_1s_0s_1s_0) \big).
\end{equation*}
The homomorphism $\phi$ is surjective since it is surjective on the generators of $Q$. Then, since $|W^B_{m_1} \times W^B_{m_2}|=|Q|$, it follows that $\phi$ is bijective and we have a group isomorphism $W^B_{m_1} \times W^B_{m_2} \cong Q$.
\end{proof}
\begin{proposition} \label{proposition:w tensor w iso to ewe}
There is a \textbf{k}-algebra isomorphism $\w_{\nu_1} \otimes_{\textbf{k}} \w_{\nu_2} \cong \textbf{e}\vv\textbf{e}$.
\end{proposition}
\begin{proof}
Define a map $\psi:\w_{\nu_1} \otimes_{\textbf{k}} \w_{\nu_2} \longrightarrow \textbf{e}\vvlm\textbf{e} $ by,
\begin{equation*}
\begin{split}
\idemp \otimes \idempj &\mapsto \idempij \\
x_k\idemp \otimes \idempj &\mapsto x_k\idempij \quad 1\leq k \leq m_1\\
\sigma_k \idemp \otimes \idempj &\mapsto \sigma_k\idempij \quad 1\leq k \leq m_1-1\\
\pi \idemp \otimes \idempj &\mapsto \pi\idempij \\
\idemp \otimes x_k\idempj &\mapsto x_{m_1+k}\idempij \quad 1\leq k \leq m_2\\
\idemp \otimes \sigma_k\idempj &\mapsto \sigma_{m_1+k}\idempij \quad 1\leq k\leq m_2-1\\
\idemp \otimes \pi\idempj &\mapsto \sigma_{m_1}\cdots \sigma_1\pi \sigma_1 \cdots \sigma_{m_1}\idempij,
\end{split}
\end{equation*}
extending \textbf{k}-linearly and multiplicatively. Then, by inspection of the defining relations, one can see that $\psi$ is well-defined and is therefore a morphism of \textbf{k}-algebras. Let $Q \subset W^B_m$ be the subgroup of $W^B_m$ as in Lemma \ref{lemma: group isomorphism}. For each $w \in Q$, fix a reduced expression and consider
\begin{equation*}
\mathcal{B}^\prime=\{ \sigma_w x_1^{n_1}\cdots x_m^{n_m}\idempij \mid w \in Q, n_k \in \mathbb{N}_0 \hspace{0.5em} \forall k, \textbf{i} \in {^\theta}I^{\nu_1}, \textbf{j} \in {^\theta}J^{\nu_2} \}.
\end{equation*}
This set is linearly independent because it is a subset of the basis given for VV algebras in Lemma \ref{vv basis}. $\mathcal{B}^\prime$ spans $\textbf{e}\w_{\nu_1+\nu_2}\textbf{e}$ because the $w \in Q$ are precisely the elements of $W^B_m$ that permute the $(i_1, \ldots, i_{m_1})$ and the $(j_1, \ldots, j_{m_2})$, and which do not intertwine elements $i_r \in I$ with elements $j_t \in J$. So $\mathcal{B}^\prime$ is a \textbf{k}-basis for $\textbf{e}\w_{\nu_1+\nu_2}\textbf{e}$.
\\
\\
We can now calculate the graded dimension of these algebras and show that they are indeed equal. First we calculate the graded dimension of $\textbf{e}\vv\textbf{e}$ using $\mathcal{B}^\prime$. The polynomial part of this basis contributes a factor of $\frac{1}{(1-q^2)^{m}}$ to this graded dimension. Then,
\begin{equation*}
\textrm{dim}_q(\textbf{e}\vv\textbf{e})=\frac{1}{(1-q^2)^{m_1+m_2}}\sum_{\substack{w \in Q \\ \textbf{i}\in {^\theta}I^{\nu_1} \\ \textbf{j} \in {^\theta}J^{\nu_2}}} q^{\textrm{deg}(\sigma_w\idempij)}.
\end{equation*}
Consider $\sigma_w \idempij$, for some $w \in Q$. Note that
\begin{equation*}
\begin{array}{l l l}
\mathrm{deg}(\sigma_{s_{m_1}\cdots s_1s_0s_1 \cdots s_{m_1}}\idempij) & = & \mathrm{deg}(\sigma_{s_0}\idempj) \\
\mathrm{deg}(\sigma_{s_{m_1+1}}\idempij) & = & \mathrm{deg}(\sigma_{s_1}\idempj) \\
 & \vdots & \\
\mathrm{deg}(\sigma_{s_{m_1+m_2-1}}\idempij) & = & \mathrm{deg}(\sigma_{s_{m_2-1}}\idempj). 
\end{array}
\end{equation*}
Similarly, $\mathrm{deg}(\sigma_{s_i}\idempij)=\mathrm{deg}(\sigma_{s_i}\idemp)$ for all $0 \leq i \leq m_1-1$. Hence,
\begin{equation*}
\mathrm{deg}(\sigma_w\idempij)=\mathrm{deg}(\sigma_u\idemp) + \mathrm{deg}(\sigma_v\idempj)
\end{equation*}
for some $u \in W^B_{m_1}$, $v \in W^B_{m_2}$. Then,
\begin{equation*}
\sum_{\substack{w \in Q \\ \textbf{i}\in {^\theta}I^{\nu_1} \\ \textbf{j} \in {^\theta}J^{\nu_2}}} q^{\textrm{deg}(\sigma_w\idempij)} = \sum_{\substack{u \in W^B_{m_1}, v \in W^B_{m_2} \\\textbf{i} \in {^\theta}I^{\nu_1} \\ \textbf{j} \in {^\theta}J^{\nu_2}}} q^{\textrm{deg}(\sigma_u\idemp) + \textrm{deg}(\sigma_v\idempj)}.
\end{equation*}
On the other hand, for $\w_{\nu_1} \otimes \w_{\nu_2}$,
\begin{equation*}
\begin{split}
\textrm{dim}_q(\w_{\nu_1}\otimes \w_{\nu_2})&=\frac{1}{(1-q^2)^{m_1}}\sum_{\substack{u \in W^B_{m_1} \\ \textbf{i} \in {^\theta}I^{\nu_1}}} q^{\textrm{deg}(\sigma_u\idemp)}\frac{1}{(1-q^2)^{m_2}}\sum_{\substack{v \in W^B_{m_2} \\ \textbf{j} \in {^\theta}J^{\nu_2}}} q^{\textrm{deg}(\sigma_v\idempj)} \\
&=\frac{1}{(1-q^2)^{m_1+m_2}}\sum_{\substack{u \in W^B_{m_1}, v \in W^B_{m_2} \\\textbf{i} \in {^\theta}I^{\nu_1} \\ \textbf{j} \in {^\theta}J^{\nu_2}}} q^{\textrm{deg}(\sigma_u\idemp) + \textrm{deg}(\sigma_v\idempj)}.
\end{split}
\end{equation*}
Hence, we have shown
\begin{equation*}
\textrm{dim}_q(\textbf{e}\vv\textbf{e})=\textrm{dim}_q(\w_{\nu_1} \otimes \w_{\nu_2}).
\end{equation*} 
To prove the claimed result it now suffices to prove that $\psi$ is surjective. 
\\
\\
Rename the generators of $Q$ as follows. Put
\begin{equation*}
c_{i+1} = \left\{
\begin{array}{l l}
s_{m_1} \cdots s_1s_0s_1 \cdots s_{m_1} & \quad i=m_1\\
s_i & \quad i\neq m_1
\end{array} \right. 
\end{equation*}
so that $Q$ is generated by $\{ c_1, c_2, \ldots, c_{m_1+m_2} \}$. Define $\ell_Q: Q \longrightarrow \mathbb{N}_0$ in the following way. For the identity element $1 \in Q$ put $\ell_Q(1)=0$. Any $w \in Q$ can be written as a product $w=c_{i_1}\cdots c_{i_k}$. Pick these generators in such a way that $k$ is minimal. Then $\ell_Q(w)=k$. For example, $\ell_Q(c_i)=1$ for every $i$. We say that $w=c_{i_1}\cdots c_{i_k}\in Q$ is an $\ell_Q$-reduced expression for $w$ if $\ell_Q(w)=k$. 
\\
\\
Let $\ell: W^B_m \longrightarrow \mathbb{N}_0$ be the usual length function on $W^B_m$. The isomorphism $\phi$ from Lemma \ref{lemma: group isomorphism} demonstrates that any $\ell_Q$-reduced expression $c_{i_1}\cdots c_{i_k}$ is also a reduced expression with respect to $\ell$. For every $w \in Q$ fix an $\ell_Q$-reduced expression $w=c_{i_1}\cdots c_{i_k}$. Then $\sigma_w \idempij=\sigma_{c_{i_1}}\cdots \sigma_{c_{i_k}}\idempij$ for each $\textbf{i} \in {^\theta}I^{\nu_1}, \textbf{j} \in {^\theta}J^{\nu_2}$.
\\
\\
It is clear that $\psi$ is surjective on elements $\idempij$ and $x_i \idempij$, for all $\textbf{i} \in {^\theta}I^{\nu_1}$, $\textbf{j} \in {^\theta}J^{\nu_2}$. Notice that 
\begin{equation*}
\begin{split}
\psi(\sigma_k\idemp\otimes\idempj)&=\sigma_{c_{k+1}}\idempij \quad \textrm{for } 0\leq k \leq m_1-1\\
\psi(\idemp\otimes \sigma_k\idempj)&=\sigma_{c_{m_1+k+1}}\idempij \quad \textrm{for } 0\leq k \leq m_2-1.
\end{split}
\end{equation*}
Therefore $\psi$ is surjective on the basis $\mathcal{B}^\prime$ and hence on $\textbf{e}\w_{\nu_1+\nu_2}\textbf{e}$.
\end{proof}
\begin{lemma} \label{lemma:e full in the separated VV alg}
Let $\textbf{e}=\sum_{\substack{\textbf{i} \in {^\theta}I^{\nu_1} \\ \textbf{j} \in {^\theta}J^{\nu_2}}} \textbf{e}(\textbf{i}\textbf{j})$. Then $\textbf{e}$ is full in the separated VV algebra $\vv$.
\end{lemma}
\begin{proof}
We must show $\vv=\vv \textbf{e} \vv$. We take any idempotent $\textbf{e}(\textbf{k}) \in \vv$, not a summand of $\textbf{e}$, and show that $\textbf{e}(\textbf{k}) \in \vv \textbf{e} \vv$. Suppose $k_1=i \in I$ (the same argument holds if $k_1=j \in J$). Let $\varepsilon_1, \ldots, \varepsilon_r$ denote the positions of entries belonging to $I$, and assume $\varepsilon_1 < \cdots < \varepsilon_r$. Since $k_1 =i \in I$, we have $\varepsilon_1=1$. Let $w_1=s_2s_3s_4 \cdots s_{\varepsilon_2-1} \in \sm$, where the $s_j$ are the generators of $\sm$. Then,
\begin{equation*}
\sigma^\rho_{w_1}\sigma_{w_1}\textbf{e}(\textbf{k})=\textbf{e}(\textbf{k})
\end{equation*}
because $\sigma_r^2\idemp = \idemp$ when there are no arrows between $i_r$ and $i_{r+1}$.
\\
\\
Suppose $\sigma_{w_1} \textbf{e}(\textbf{k})=\textbf{e}(\textbf{k}_1) \sigma_{w_1}$. Then the first two entries of $\textbf{k}_1$ are elements of $I$. Let $w_2=s_3s_4s_5\cdots s_{\varepsilon_3-1} \in \sm$. Then,
\begin{equation*}
\sigma^\rho_{w_2}\sigma_{w_2}\textbf{e}(\textbf{k}_1)=\textbf{e}(\textbf{k}_1)
\end{equation*}
for the same reasoning as above. Suppose $\sigma_{w_2} \textbf{e}(\textbf{k}_1)=\textbf{e}(\textbf{k}_2) \sigma_{w_2}$. Then the first three entries of $\textbf{k}_2$ are elements of $I$.
\\
\\
Continuing like this we obtain $\sigma_{w_2}, \ldots, \sigma_{w_{r-1}}$ with $\sigma_{w_t}^\rho \sigma_{w_t} \textbf{e} (\textbf{k}_{t-1}) = \textbf{e} (\textbf{k}_{t-1})$, for $t$ with $2 \leq t \leq r-1$. Then we have,
\begin{equation*}
\sigma_{w_1}^\rho \sigma_{w_2}^\rho \cdots \sigma_{w_{r-1}}^\rho \idempij \sigma_{w_{r-1}} \cdots \sigma_{w_2}\sigma_{w_1}\textbf{e}(\textbf{k}) = \textbf{e}(\textbf{k}),
\end{equation*}
and $\idempij$ is a summand of $\textbf{e}$. Hence $\textbf{e}(\textbf{k}) \in \vv \textbf{e} \vv$ so that $\vv = \vv \textbf{e} \vv$, as required.
\end{proof}
\begin{corollary} \label{theorem:ME in the separated case}
$\vv$ and $\w_{\nu_1} \otimes_{\textbf{k}} \w_{\nu_2}$ are Morita equivalent.
\end{corollary}
\begin{proof}
Using Lemma \ref{lemma:e full in the separated VV alg}, we have that $\vv\textbf{e}$ is a progenerator in $\vv$-Mod. To show Morita equivalence it remains to show $\textrm{End}_{\vv}(\vv\textbf{e}) \cong \mathfrak{W}_{\nu_1} \otimes \mathfrak{W}_{\nu_2}$. But, by Proposition \ref{proposition:w tensor w iso to ewe} and the fact that $\textrm{End}_{\vv}(\vv\textbf{e}) \cong \textbf{e}\vv\textbf{e}$, we are done.
\end{proof}
\begin{remark} 
In this section we have defined families of VV algebras arising from unions of $\mathbb{Z} \rtimes \mathbb{Z}_2$-orbits. But in fact Corollary \ref{theorem:ME in the separated case} shows that in order to study these VV algebras it suffices to study the families of VV algebras arising from a single $\mathbb{Z} \rtimes \mathbb{Z}_2$-orbit, as we defined in Section \ref{section:vv definition}. A slightly modified version of this theorem, using the same proof, explains why we may assume that $\nu$ has connected support.
\end{remark}
\subsection{Morita Equivalence in the Case $p,q \not\in I$} \label{section:Morita equivalence}
In this section we assume $p,q \not\in I$. Then $I=I_\lambda$ is the $\mathbb{Z} \rtimes \mathbb{Z}_2$-orbit of $\lambda \in \textbf{k}^\times$, for some $\lambda \neq p,q$.
\\
\\
The defining relations of $\vv$ are dependent on whether or not $q$ is an element of $I$. In particular, in this case, for every idempotent $\idemp$, we have
\begin{equation*}
\begin{split}
\pi^2\idemp &=\idemp \\
\pi \sigma_1 \pi \sigma_1 \idemp &= \sigma_1 \pi \sigma_1 \pi \idemp \\
\textrm{deg}(\pi\idemp)&=0
\end{split}
\end{equation*}
For any VV algebra $\vv$ we can consider various idempotent subalgebras $\textbf{e}\vv\textbf{e}$, for different choices of $\textbf{e}$. Each of these idempotent subalgebras may or may not be isomorphic to a KLR algebra (see Proposition \ref{proposition:klr an idempotent subalgebra}). Among these KLR algebras we can always distinguish $\subklr^+$ and $\subklr^-$, as mentioned in Remark \ref{remark:every VV has KLR+ as idemp subalg}. Here we show that the VV algebras arising from the setting $p,q \not\in I$ are Morita equivalent to KLR algebras of type $A$. Namely, for any $\nu \in {^\theta}\mathbb{N}I$, $\vv$ and $\subklr^+$ are Morita equivalent, as are $\vv$ and $\subklr^-$. We show this here for $\subklr^+$. Let $|\nu|=2m$. Again, we are using the notation as in Remark \ref{definition:KLR^+ and KLR^-}. To stress this point; for $\nu = \sum_{i \in I^+}\nu_ii + \sum_{i \in I^-}\nu_ii$ we set $\tilde{\nu}= \sum_{i \in I^+}\nu_ii \in \mathbb{N}I^+$.
\begin{theorem} \label{theorem:VV Morita equiv to KLR}
$\vv$ and $\subklr^+$ are Morita equivalent.
\end{theorem}
\begin{proof}
It suffices to find a progenerator $P \in \vv$-Mod such that $\subklr^+ \cong\textrm{End}_{\vv}(P)$.
\\
\\
Let 
\begin{equation*}
\textbf{e}:=\sum_{\textbf{i} \in I^{\tilde{\nu}}} \idemp
\end{equation*}
Now, $\vv\textbf{e}$ is a progenerator if and only $\textbf{e}$ is full in $\vv$, i.e. if and only if $\vv\textbf{e}\vv=\vv$.
\\
\\
Clearly $\vv\textbf{e}\vv\subseteq\vv$. So it remains to show $\vv\subseteq\vv\textbf{e}\vv$. We do this by showing that every idempotent $\idemp$ lies in $\vv\textbf{e}\vv$. Then, since $\sum_{\textbf{i} \in {^\theta}I^\nu} \idemp =1$, it follows that $1 \in \vv\textbf{e}\vv$ and so $\vv\textbf{e}\vv \subseteq \vv$. 
\\
\\
If $\idemp$ is a summand of $\textbf{e}$ then it is clear that $\idemp \in \vv \textbf{e} \vv$. Take $\idemp \in \vv$ not a summand of $\textbf{e}$. Then there are finitely many entries of $\textbf{i}$, say $i_{k_1}, \ldots, i_{k_r}$, with $i_{k_s} \in I^-$ for all $1\leq s \leq r$, and with $k_1 < k_2 < \cdots < k_r$. Let $\eta$ be the minimal length left coset representative of $\sm$ in $W^B_m$ given by
\begin{equation*}
\eta=s_{k_1-1}\cdots s_1s_0 \cdots \cdots s_{k_{r-1}-1}\cdots s_1s_0 s_{k_r-1}\cdots s_1s_0
\end{equation*}
Fix this reduced expression of $\eta$. Fix the reduced expression
\begin{equation*}
\eta^{-1}=s_0 s_1 \cdots s_{k_r-1} s_0 s_1 \cdots s_{k_{r-1}-1} \cdots \cdots s_0s_1 \cdots s_{k_1-1}. 
\end{equation*}
Then $\idemp\sigma_\eta = \sigma_\eta \idempj$ for some $\idempj$, where $\textbf{j} \in I^{\tilde{\nu}}$ so that $\idempj$ is a summand of $\textbf{e}$, and
\begin{equation*}
\sigma_\eta \idempj\sigma_{\eta^{-1}}=\idemp.
\end{equation*}
Hence $\idemp \in \vv\textbf{e}\vv$ as required. It follows now that $\textbf{e}$ is full in $\vv$ so that $\vv\textbf{e}$ is a progenerator in $\vv$-Mod.
\\
\\
It remains to show $\subklr^+ \cong \textrm{End}_{\vv}(\vv\textbf{e})$. But, \hbox{$\textrm{End}_{\vv}(\vv\textbf{e}) \cong \textbf{e}\vv\textbf{e}$} and, by Proposition \ref{proposition:klr an idempotent subalgebra}, $\textbf{e}\vv\textbf{e} \cong \subklr^+$. It follows that $\vv$ and $\subklr^+$ are Morita equivalent.
\end{proof}
Note that the dual proof shows Morita equivalence between $\vv$ and $\subklr^-$, for $\tilde{\nu} \in \mathbb{N}I^-$.
\subsection*{A Note on the Cases: $p\in I$, $q\in I$} \label{subsection:a note on the cases p,q in I}
\subsubsection*{Case: $\bf{q \in I}$}
Consider first the case $q\in I$, $p \not\in I$. Take $\nu \in {^\theta}\mathbb{N}I_q$ and consider the associated VV algebra $\vv$. Note that many of the relations now depend on whether or not $i_1=q^{\pm 1}$. If we pick $\nu$ so that $q$ is not a summand then of course we always have $i_1 \neq q^{\pm 1}$. Hence, in this case, the relations are exactly the same as those in the ME setting so that $\vv$ is Morita equivalent to $\subklr^+$. From now on, when we work in the setting $q\in I$, we assume $q$ (and hence $q^{-1}$) is a summand of $\nu$. That is, $\nu_q \geq 1$. 
\subsubsection*{Case: $\bf{p \in I}$}
Now consider the case $p\in I$, $q \not\in I$. Take $\nu \in {^\theta}\mathbb{N}I_p$ and the associated VV algebra $\vv$. The defining relations do not explicitly depend on whether or not $p \in I$, but the subtle difference arises when we examine the underlying quiver $\Gamma_{I_p}$. Locally, with regards to the relations, $\Gamma_{I_p}$ is exactly the same as $\Gamma_{I_\lambda}$, for some $\lambda\neq p,q$, except in the following neighbourhood of $\Gamma_{I_p}$.
\begin{center}
\begin{tikzpicture}
\matrix (m) [matrix of math nodes, row sep=2.5em,
column sep=2em, text height=1.5ex, text depth=0.25ex]
{ \cdots & p & p^{-1} & \cdots \\};
\path[->]
(m-1-1) edge  (m-1-2)
(m-1-2) edge  (m-1-3)
(m-1-3) edge  (m-1-4);
\end{tikzpicture}
\end{center}
In other words, $\Gamma_{I_p^+}$ and $\Gamma_{I_p^-}$ are not two disjoint connected components of $\Gamma_{I_p}$. If we choose $\nu \in {^\theta}\mathbb{N}I$ with $\nu_p \leq 1$ the defining relations of $\vv$ are exactly those in the ME setting. So again, in this case, $\vv$ and $\subklr^+$ are Morita equivalent. From now on, when we work in the setting $p \in I$, we assume $\nu_p \geq 2$ (and hence $\nu_{p^{-1}} \geq 2$). 
\subsection{Morita Equivalences in the Case $q \in I$}
For the remainder of this section, unless stated otherwise, we assume that $p$ is not a root of unity.
\subsubsection{Affine Cellularity of Classes of VV Algebras} \label{subsection:affine cellularity}
Graham and Lehrer defined the notion of cellularity for finite-dimensional algebras in \cite{GrahamLehrer}. They defined these algebras in terms of a basis satisfying various combinatorial properties. Showing that an algebra is cellular gives rise to a parametrisation of its irreducible modules. The notion of affine cellularity was introduced by Koenig and Xi in \cite{KoenigXi}, and extends the notion of cellularity to algebras which need not be finite-dimensional.
\begin{defn}
A $\textbf{k-involution}$ of a $\textbf{k}$-algebra $A$ is a $\textbf{k}$-linear anti-automorphism $\omega$ with $\omega^2=\textrm{id}_A$.
\end{defn}
By an \textbf{affine algebra} we mean a commutative $\textbf{k}$-algebra $B$ which is a quotient of a polynomial ring $\textbf{k}[x_1,\ldots,x_n]$ in finitely many variables.
\begin{defn}[\cite{KoenigXi}, Definition 2.1] \label{definition:affine cell ideal}
Let $A$ be a unitary $\textbf{k}$-algebra with a $\textbf{k}$-involution $\omega$ on $A$. A two-sided ideal $J \subseteq A$ is called an \textbf{affine cell ideal} if and only if the following data are given and the following conditions are satisfied.
\begin{enumerate}
\item $\omega(J)=J$.
\item There exists a free $\textbf{k}$-module $V$ of finite rank and an affine commutative $\textbf{k}$-algebra $B$ with identity and with a $\textbf{k}$-involution $i$ such that $\Delta:=V\otimes_{\textbf{k}}B$ is an $A$-$B$-bimodule, where the right $B$-module structure is induced by that of the right regular $B$-module.
\item There is an $A$-$A$-bimodule isomorphism $\alpha: J \longrightarrow \Delta \otimes_B \Delta^\prime$, where $\Delta^\prime = B \otimes_\textbf{k} V$ is a $B$-$A$-bimodule with the left $B$-structure induced by the left regular $B$-module. The right $A$-structure is induced via $\omega$. That is,
\begin{equation*}
(b\otimes v)a:=s(\omega(a)(v\otimes b)) \textrm{ for } a\in A, b \in B, v \in V
\end{equation*}
where $s:V \otimes B \longrightarrow B \otimes V$, $v \otimes b \mapsto b \otimes v$ is the switch map, such that the following diagram commutes:
\begin{center}
\begin{tikzpicture}
\matrix (m) [matrix of math nodes,row sep=3.5em, column sep=4.0em,]
{J & \Delta \otimes_B \Delta^\prime \\
J & \Delta \otimes_B \Delta^\prime \\};

\path[-stealth] (m-1-1) edge node [above] {$\alpha$} (m-1-2);
\path[-stealth] (m-1-1) edge node [left] {$\omega$} (m-2-1);
\path[-stealth] (m-2-1) edge node [below] {$\alpha$} (m-2-2);
\path[-stealth] (m-1-2) edge node [right] {$v\otimes b \otimes b^\prime\otimes w \mapsto w\otimes i(b^\prime)\otimes i(b)\otimes v$} (m-2-2);
	
\end{tikzpicture}
\end{center}
\end{enumerate}
\end{defn}
\begin{defn}[\cite{KoenigXi}, Definition 2.1] \label{definition:affine cellular algebra}
A \textbf{k}-algebra $A$, with a \textbf{k}-involution $\omega$, is called \textbf{affine cellular} if there is a $\textbf{k}$-module decomposition $A=J_1^\prime \oplus \cdots \oplus J_n^\prime$, for some $n$, with $\omega(J^\prime_j)=J^\prime_j$ for each $j$ and such that setting $J_j=\bigoplus_{l=1}^jJ_l^\prime$ gives a chain of two-sided ideals of $A$,
\begin{equation*}
(0)=J_0 \subset J_1 \subset \cdots \subset J_n=A,
\end{equation*}
and for each $j$ the quotient $J_j^\prime \cong J_j/J_{j-1}$ is an affine cell ideal of $A/J_{j-1}$ (with respect to the involution induced by $\omega$ on the quotient). This chain is called a \textbf{cell chain} for the affine cellular algebra $A$. The module $\Delta$ is called a \textbf{cell lattice} for the affine cell ideal $J$.
\end{defn}
\begin{remark}
In \cite{KleshchevAffineQH}, Kleshchev gives graded versions of Definitions \ref{definition:affine cell ideal} and \ref{definition:affine cellular algebra} in which all algebras, ideals, etc. are graded and the maps $\omega$, $i$ are homogeneous. 
\end{remark}
In this Section we fix the following setting. Assume $q \in I$, $p \not\in I$, $p$ not a root of unity and take $\nu \in {^\theta}\mathbb{N}I_q$ with multiplicity one and $|\nu|=2m$, for some $m \in \mathbb{N}$. Suppose we have fixed a reduced expression $s_{i_1}\cdots s_{i_k}$ for some $w \in W^B_m$ so that $\sigma_w=\sigma_{i_1}\cdots \sigma_{i_k}$. Define
\begin{equation*}
\sigma^\rho_w:=\sigma_{i_k} \cdots \sigma_{i_1}.
\end{equation*}
Let $\Pi^+$ denote the root partitions of the idempotent subalgebra $\subklr^+ \subseteq \vv$ and let $\Pi^-$ denote the root partitions of the idempotent subalgebra $\subklr^- \subseteq \vv$. Put $\Pi^\pm=\Pi^+ \cup\Pi^-$. Throughout this section let $\textbf{e}:=\sum_{\lambda \in \Pi^\pm} \eilambda{\lambda}$ where $\textbf{e}(\textbf{i}_\lambda)$ denotes the idempotent associated to the root partition $\lambda \in \Pi^\pm$, as discussed in \ref{subsection:klr root partitions}. 
\\
\\
Define $\Pi(m)$ to be the following set. $\Pi(m):=\big\{ (a_1, \ldots, a_{m-1}) \mid a_i \in \{1,2\} \hspace{0.3em} \forall \hspace{0.1em} i \big\}$. There is a bijection,
\begin{equation} \label{equation:bijection theta}
\begin{split}
\theta : \Pi^+ &\longrightarrow \Pi(m) \\
\lambda &\mapsto (a_1, \ldots, a_{m-1})
\end{split}
\end{equation}
where $a_i=\left\{
\begin{array}{l l}
    1 & \hspace{0.1em} \textrm{if }p^{2i-2}q\textrm{ appears before }p^{2i}q \textrm{ in } \textbf{i}_\lambda \\
    2 & \hspace{0.1em} \textrm{if }p^{2i-2}q\textrm{ appears after }p^{2i}q \textrm{ in } \textbf{i}_\lambda.
\end{array} \right. $
\\
\\
Similarly there is a one-to-one correspondence between $\Pi^-$ and $\Pi(m)$.
\\
\\
The next lemma states that every idempotent $\idemp \in \vv$ is isomorphic to either an idempotent in $\subklr^+$ or to an idempotent in $\subklr^-$.
\begin{lemma} \label{lemma:isomorphic idemps}
Take $\idemp \in \vv$ with $\textbf{i} \notin I^{\tilde{\nu}^+} \cup I^{\tilde{\nu}^-}$. Note that precisely one of $q,q^{-1}$ will appear as an entry in $\textbf{i}$.
\begin{itemize}
\item[(i)] If $q$ is in position $k$ (i.e. $i_k = q$) then $\idemp \cong \idempj$ for some $\textbf{j} \in I^{\tilde{\nu}^+}$. 
\item[(ii)] If $q^{-1}$ is in position $k$ (i.e. $i_k=q^{-1}$) then $\idemp \cong \idempj$ for some $\textbf{j} \in I^{\tilde{\nu}^-}$.
\end{itemize}
\end{lemma}
\begin{proof} 
For $(i)$, assume $q$ is in position $k$, i.e. $i_k=q$. We require $a \in \idemp \vv \idempj$, $b \in \idempj \vv \idemp$ such that $ab=\idemp$ and $ba=\idempj$, for some $\textbf{j} \in I^{\tilde{\nu}^+}$. 
\\
\\
Since $\textbf{i} \notin I^{\tilde{\nu}^+} \cup I^{\tilde{\nu}^-}$ there is a non-empty subset $\{ \varepsilon_1, \ldots, \varepsilon_d \} \subsetneq \{ 1, \ldots, m \}$, with $\varepsilon_r < \varepsilon_{r+1}$ for all $r$, such that $i_{\varepsilon_r}=p^{2n_r}q^{-1}$, where $n_r \in \mathbb{Z}\setminus \{0\}$ for all $r$. Note also that $\varepsilon_r \neq k$ for all $r$ since $i_k=q$. Put,
\begin{equation*}
\begin{split}
&a=\idemp\sigma_{\varepsilon_1-1}\cdots \sigma_1\pi\sigma_{\varepsilon_2-1}\cdots \sigma_1\pi \cdots \sigma_{\varepsilon_d-1} \cdots \sigma_1\pi\idempj \in \idemp\vv\idempj \\
&b=\idempj\pi\sigma_1\cdots \sigma_{\varepsilon_d-1}\pi \sigma_1\cdots \sigma_{\varepsilon_{d-1}-1}\cdots \pi\sigma_1\cdots\sigma_{\varepsilon_1-1}\idemp \in \idempj\vv\idemp.
\end{split}
\end{equation*}
Then $\idempj \in \vv$ is such that $\textbf{j} \in I^{\tilde{\nu}^+}$ and, from the relations, we know that $ab=\idemp$ and $ba=\idempj$ so that $\idemp\cong \idempj$. A similar argument is used for $(ii)$ when $i_k = q^{-1}$.
\end{proof}
\begin{corollary} \label{corollary:every idemp is iso to a root partition idemp}
Every idempotent $\idemp \in \vv$ is isomorphic to some $\textbf{e}(\textbf{j}_\lambda)$, $\lambda \in \Pi^{\pm}$.
\end{corollary}
\begin{proof}
By Lemma \ref{lemma:isomorphic idemps}, it suffices to prove that every $\idemp$, with $\textbf{i} \in I^{\tilde{\nu}^+}$ or with $\textbf{i} \in I^{\tilde{\nu}^-}$, is isomorphic to some $\textbf{e}(\textbf{j}_\lambda)$, $\lambda \in \Pi^{\pm}$. Suppose $\textbf{i} \in I^{\tilde{\nu}^+}$ (the proof for when $\textbf{i} \in I^{\tilde{\nu}^-}$ is the same). Let $|\nu|=2m$ so that $|\tilde{\nu}^+|=m$.
\\
\\
Take any $\textbf{i} \in I^{\tilde{\nu}^+}$. Associate to $\textbf{i}$ the $(m-1)$-tuple $(a_1, \ldots, a_{m-1}) \in \Pi(m)$, where 
\begin{equation*}
a_j=\left\lbrace  
\begin{array}{l l}
1 & \textrm{ if } p^{2j-2}q \textrm{ appears before } p^{2j}q \textrm{ in } \textbf{i} \\
2 & \textrm{ if } p^{2j-2}q \textrm{ appears after } p^{2j}q \textrm{ in } \textbf{i}.
\end{array}
\right.
\end{equation*}
Then there exists a surjection,
\begin{equation*}
\begin{split}
g:I^{\tilde{\nu}^+} &\twoheadrightarrow \Pi(m) \\
\textbf{i} &\mapsto (a_1, \ldots, a_m).
\end{split}
\end{equation*}
We have seen already from \ref{equation:bijection theta} that there is a one-to-one correspondence between $\Pi(m)$ and root partitions associated to $\tilde{\nu}$. Now, $g(\textbf{i}) \in \Pi(m)$ and therefore corresponds to some root partition $\lambda \in \Pi^+$, i.e. $g(\textbf{i})=\theta(\lambda)$. This means we have $\idemp \cong \eilambda{\lambda}$, because when we permute entries of $\textbf{i}$ which do not affect $g(\textbf{i})$, we are permuting entries of $\textbf{i}$ in such a way that neighbouring vertices of $\Gamma_I$ never switch position, by definition. Hence we can permute entries of $\textbf{i}$ in this way to obtain a root partition, $\eilambda{\lambda}$. In other words, there exists an element $w \in \sm$ such that $\sigma_w^\rho \eilambda{\lambda} \sigma_w \idemp = \idemp$.
\end{proof}
\begin{corollary} \label{corollary:e is full in vv}
The idempotent $\textbf{e}=\sum_{\lambda \in \Pi^\pm} \eilambda{\lambda}$ is full in $\vv$.
\end{corollary}
\begin{proof}
It is clear that $\vv\textbf{e}\vv \subseteq \vv$. To show $\vv \subseteq \vv\textbf{e}\vv$ it suffices to show $\idemp \in \vv\textbf{e}\vv$, for every idempotent $\idemp$ not a summand of $\textbf{e}$. Let $\idemp$ be such an idempotent. By Corollary \ref{corollary:every idemp is iso to a root partition idemp}, \hbox{$\idemp \cong \textbf{e}(\textbf{i}_\lambda)$} for some $\lambda \in \Pi^{\pm}$. That is, there exists $a \in \idemp\vv \eilambda{\lambda}$, $b \in \eilambda{\lambda} \vv \idemp$ such that 
\begin{equation*}
\idemp=ab=a\eilambda{\lambda} b \in \vv\textbf{e}\vv.
\end{equation*}
Then $\vv=\vv\textbf{e}\vv$ as required.
\end{proof}
Let $A$ and $\tilde{A}$ be the isomorphic path algebras of the following quivers, respectively.
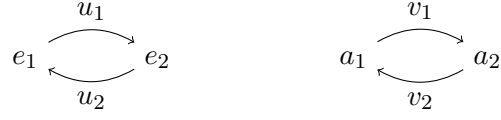
\begin{center}
\begin{table}[h]
\begin{tabular}{cc}
\begin{minipage}{4cm}
\begin{tikzpicture}
\matrix (m) [matrix of math nodes, row sep=3em,
column sep=3em, text height=1.5ex, text depth=0.25ex]
{e_1 & e_2 \\};
\path[->]
(m-1-1) edge [bend left=30] node [above] {$u_1$} (m-1-2)
(m-1-2) edge [bend left=30] node [below] {$u_2$} (m-1-1);
\end{tikzpicture}
\end{minipage}
&
\begin{minipage}{4cm}
\begin{tikzpicture}
\matrix (m) [matrix of math nodes, row sep=3em,
column sep=3em, text height=1.5ex, text depth=0.25ex]
{a_1 & a_2 \\};
\path[->]
(m-1-1) edge [bend left=30] node [above] {$v_1$} (m-1-2)
(m-1-2) edge [bend left=30] node [below] {$v_2$} (m-1-1);
\end{tikzpicture}
\end{minipage}
\end{tabular}
\end{table}
\end{center}
$A$ and $\tilde{A}$ are graded $\textbf{k}$-algebras via
\begin{equation*}
\begin{split}
&\textrm{deg}(e_i)=\textrm{deg}(a_i)=0 \\
&\textrm{deg}(u_je_j)=\textrm{deg}(v_ja_j)=1
\end{split}
\end{equation*}
for $i,j \in \{ 1,2 \}$. 
\begin{theorem} \label{theorem:ME in case q in I mult one}
$\vv$ and $A^{\otimes (m-1)}\otimes \tilde{A}$ are Morita equivalent.
\end{theorem}
\begin{proof}
Define a map 
\begin{equation*}
\phi:A^{\otimes (m-1)}\otimes \tilde{A} \longrightarrow \textbf{e}\vv\textbf{e}
\end{equation*}
as follows.
\begin{equation*}
\begin{array}{l}
\begin{array}{l l}
e_{j_1}\otimes \cdots \otimes e_{j_{m-1}}\otimes a_i \mapsto \eilambda{\lambda} & \textrm{for } \lambda \in \left\{
\begin{array}{l l}
    \Pi^+ \textrm{ if } i=1 \\
    \Pi^- \textrm{ if } i=2
\end{array} \right. \\
 & \textrm{and }\theta(\lambda)=(j_1,\ldots , j_{m-1})
\end{array} \\ \\
\begin{array}{l l}
e_{j_1}\otimes \cdots \otimes e_{j_k^\prime}u_\ell e_{j_k} \otimes \ldots \otimes e_{j_{m-1}}\otimes a_i \mapsto \eilambda{\lambda^{'}}\sigma_w \eilambda{\lambda} & \textrm{ for the unique} \\  &\textrm{ }w \in \sm \textrm{ such that } \\ 
 & \textrm{ }\theta(\lambda)=(j_1, \ldots , j_k, \ldots , j_{m-1}) \\
 & \textrm{ }\theta(\lambda^\prime)=(j_1, \ldots , j_k^\prime, \ldots , j_{m-1})
\end{array} \\ \\
\begin{array}{l l}
e_{j_1}\otimes \cdots \otimes e_{j_{m-1}}\otimes v_i a_i \mapsto \sigma_\eta \eilambda{\lambda} & \textrm{where }\eta \in \mathcal{D}(W^B_m/\sm)\textrm{ is the longest element}\\
 & \textrm{and }\theta(\lambda)=(j_1,\ldots , j_{m-1}).
\end{array}
\end{array}
\end{equation*}
Extend this map \textbf{k}-linearly and multiplicatively. Since $\phi$ is defined on the generators of $A^{\otimes (m-1)}\otimes \tilde{A}$, and the commutativity of elements in each tensor factor is preserved, the map $\phi$ is well-defined and is an algebra morphism. 
\\
\\
We claim that $\phi$ is an isomorphism of $\textbf{k}$-algebras. We first check surjectivity. Since there is a bijection between $\Pi(m)$ and $\Pi^+$ we have that $\phi$ is surjective on idempotents $\eilambda{\lambda}$. The map $\phi$ is also surjective on elements $\eilambda{\lambda^{'}}\sigma_w \eilambda{\lambda}$, $w \in W^B_m$. To see this, take any $\eilambda{\lambda^{'}} \sigma_w \eilambda{\lambda} \in \textbf{e} \vv \textbf{e}$. We have 
\begin{equation*}
\begin{split}
\theta(\lambda)&=(j_1, \ldots, j_{m-1}) \\
\theta(\lambda^{'})&=(j_1^{'}, \ldots, j_{m-1}^{'})
\end{split}
\end{equation*}
for some $(j_1, \ldots, j_{m-1})$, $(j_1^{'}, \ldots, j_{m-1}^{'}) \in \Pi(m)$. Then
\begin{equation*}
\phi(u_{j_1}^{\varepsilon_1}e_{j_1} \otimes u_{j_2}^{\varepsilon_2}e_{j_2} \otimes \cdots \otimes u_{j_{m-1}}^{\varepsilon_{m-1}}e_{j_{m-1}} \otimes v_i^{\varepsilon_i}a_i)=\eilambda{\lambda^{'}} \sigma_w \eilambda{\lambda},
\end{equation*}
where, for $1 \leq j \leq m-1$, 
\begin{equation*}
\begin{array}{l}
\varepsilon_j=\left\lbrace
\begin{array}{l l}
1 & \textrm{if } j_1 \neq j_1^{'} \\
0 & \textrm{if } j_1 = j_1^{'} 
\end{array} \right. 
\\ \\
\varepsilon_i = \left\lbrace
\begin{array}{l l}
0 & \textrm{if } \lambda, \lambda^{'} \in \Pi^+ \textrm{ or } \lambda, \lambda^{'} \in \Pi^- \\
1 & \textrm{else}.
\end{array} \right.
\end{array}
\end{equation*}
It remains to check surjectivity on the $x_j\eilambda{\lambda}$, $1 \leq j \leq m$.\\
\\
Take any $\lambda \in \Pi^\pm$ with $\theta(\lambda)=(j_1,\ldots,j_{m-1})$. For an entry $a \in \textbf{i}_\lambda$ let $\psi(a)$ denote its position in $\textbf{i}_\lambda$. Let us assume that $q$ is an entry of $\textbf{i}_\lambda$ since the same argument holds when $q^{-1}$ is an entry of $\textbf{i}_\lambda$. Then 
\begin{equation*}
\phi(e_{j_1}\otimes \cdots \otimes e_{j_{m-1}}\otimes \pi^2 a_i)=x_{\psi(q)}\eilambda{\lambda}. 
\end{equation*}
Now consider $x_j\eilambda{\lambda}$, for any $j\neq \psi(q)$. Let $i_j$ be the $j^{th}$ entry of $\textbf{i}_\lambda$. Note that 
\begin{equation*}
x_j\eilambda{\lambda}=(x_j-x_k+x_k)\eilambda{\lambda}
\end{equation*}
where $k=\psi(p^{-2}i_j)$ is the position of $p^{-2}i_j$ in $\textbf{i}_\lambda$. 
\\
\\
But note that
\begin{equation*}
(x_j-x_k)\eilambda {\lambda} = \sigma^\rho_w  \eilambda{\lambda^\prime} \sigma_w \eilambda{\lambda}
\end{equation*}
for some $w \in \sm$, where $\lambda^\prime \in \Pi^{\pm}$ is the root partition in which simple roots appear in the same order (possibly with different position) as in $\lambda$, except for $i_j$ and $p^{-2}i_j$. In particular, $\theta(\lambda)=(j_1, \ldots , j_k, \ldots , j_{m-1})$ and $\theta(\lambda^\prime)=(j_1, \ldots , j_k^\prime, \ldots j_{m-1})$ for some $1\leq k\leq m-1$. We can repeat this until we get
\begin{equation*}
x_j\eilambda{\lambda}=(x_j-x_k+x_k-\cdots -x_\ell+x_\ell)\eilambda{\lambda} 
\end{equation*}
where $\ell=\psi(q)$. So $x_j\eilambda{\lambda}=(\sigma^\rho_{w_1}\sigma_{w_1}+\cdots + \sigma^\rho_{w_k}\sigma_{w_k} + x_{\psi(q)})\eilambda{\lambda}$. Since $\phi$ is surjective on the $\eilambda{\lambda^\prime}\sigma_w\eilambda{\lambda}$ we have $\phi$ surjective on the $x_j\eilambda{\lambda}$. 
\\
\\
We complete the proof by comparing the graded dimensions of $A^{\otimes (m-1)}\otimes \tilde{A}$ and $\textbf{e}\vv \textbf{e}$ and showing that they are equal. 
\\
\\
Consider first the graded dimension of $A$. There are two elements in degree 0: $e_1$ and $e_2$. There are two elements in degree 1: $u_1 e_1$ and $u_2 e_2$. Indeed we have two elements in each degree and so dim$_q A=2+2q+2q^2+2q^3+\cdots =\frac{2}{1-q}$. Noting that $A\cong \tilde{A}$, we therefore have
\begin{equation*}
\textrm{dim}_q(A^{\otimes (m-1)}\otimes \tilde{A})=\frac{2^m}{(1-q)^m}.
\end{equation*}
\begin{claim} dim$_q(\textbf{e} \vv \textbf{e}) = \hspace{0.1em} \textrm{dim}_q(A^{\otimes (m-1)}\otimes \tilde{A})$.
\end{claim}
For every $w \in W^B_m$ fix a reduced expression of the form $w=\eta s$, for $\eta \in \mathcal{D}(W^B_m/\sm)$ and $s \in \sm$. 
\\
\\
First we show that $\phi$ is degree-preserving. This is clear on idempotents and $x_k\textbf{e}(\textbf{i}_\lambda)$. Now take $\eilambda{\lambda^{'}}\sigma_w \eilambda{\lambda}$, for some $w \in W^B_m$.
\begin{itemize}
\item Suppose first that $\lambda, \lambda^\prime$ both lie either in $\Pi^+$ or in $\Pi^-$. Then $w=s \in \sm$. Let \hbox{$s=s_{i_1}\cdots s_{i_k}$} be the fixed reduced expression for $s$ so that $\sigma_s\textbf{e}(\textbf{i}_\lambda)=\sigma_{i_1} \cdots \sigma_{i_k}\textbf{e}(\textbf{i}_\lambda)$. Now, deg$(\sigma_{i_j}\idemp)$ is 1 if $\sigma_{i_j}$ swaps $i_j$ and $p^{\pm 2} i_j$, and is 0 otherwise. Then deg$(\sigma_s\textbf{e}(\textbf{i}_\lambda))$ is the number of pairs $(i_j,p^{\pm 2}i_j)$ in $\textbf{e}(\textbf{i}_\lambda)$ which appear in the opposite order in $\textbf{e}(\textbf{i}_{\lambda^\prime})$. Letting $(a_1, \ldots, a_{m-1})_\lambda$, $(a_1, \ldots, a_{m-1})_{\lambda^\prime}$ be the elements of $\Pi(m)$ corresponding to $\lambda$, $\lambda^\prime \in \Pi^\pm$ respectively, we find deg$(\sigma_s \textbf{e}(\textbf{i}_\lambda))$ is the number of entries in $(a_1, \ldots, a_{m-1})_\lambda$ different to entries in $(a_1, \ldots, a_{m-1})_{\lambda^\prime}$, i.e. the number of $u_\ell$ appearing as tensorands in $\phi^{-1}(\eilambda{\lambda^{'}}\sigma_w \eilambda{\lambda})$.
\item Now suppose one of $\lambda, \lambda^\prime$ lies in $\Pi^+$ and the other in $\Pi^-$. Then $w=\eta s$, where $s \in \sm$ and $\eta \in \mathcal{D}(W^B_m/\sm)$ is the longest element, and $\sigma_w \eilambda{\lambda}=\sigma_\eta \sigma_s \eilambda{\lambda}$. Since we always have deg$(\sigma_\eta \eilambda{\lambda})=1$ it follows that $\textrm{deg}(\eilambda{\lambda^{'}}\sigma_w \eilambda{\lambda})=\textrm{deg}(\sigma_\eta\eilambda{\lambda^{''}})+\textrm{deg}(\sigma_s\eilambda{\lambda})=1+\textrm{deg}(\sigma_s\eilambda{\lambda})$, for some root partition $\lambda^{''}$.
\end{itemize}
Hence $\textrm{deg}(\eilambda{\lambda^{'}}\sigma_w\eilambda{\lambda})=\#\{ u_j, v_k \textrm{ appearing in }\phi^{-1}(\sigma_w\eilambda{\lambda}) \}$ and so $\phi$ is degree-preserving.\\
\\
This means we have a bijection 
\begin{equation*} 
\{ \eilambda{\lambda^{'}}\sigma_w \eilambda{\lambda} \mid \lambda^\prime, \lambda \in \Pi^\pm \} \longleftrightarrow \{ \gamma_1 \otimes \cdots \otimes \gamma_m \in A^{\otimes (m-1)}\otimes \tilde{A}\hspace{0.3em}|\hspace{0.3em} \textrm{deg }\gamma_i \leq 1 \hspace{1em} \forall i \}.
\end{equation*}
Put $\mathscr{A}:=A^{\otimes (m-1)}\otimes \tilde{A}$ and let $\mathscr{A}_{\textrm{deg}\leq 1}$ be the vector space $\langle \gamma_1 \otimes \cdots \otimes \gamma_m \mid \textrm{deg}\gamma_i \leq 1 \rangle$. Each tensorand has two elements in each degree. Then the graded dimension of $\mathscr{A}_{\textrm{deg}\leq 1}$ has a factor of $2^m$. In each degree $k$ we choose $k$ tensorands from a possible $m$. There are $2^m\binom{m}{0}$ elements in degree 0. There are $2^m\binom{m}{1}$ elements in degree 1. In degree $k$ there are $2^m\binom{m}{k}$ elements.
\\
\\
Then,
\begin{equation*}
\textrm{dim}_q \mathscr{A}_{\textrm{deg}\leq 1}= 2^m\sum_{k=0}^m \binom{m}{k}q^k = 2^m(1+q)^m=\sum_{\lambda,\lambda^{'}}q^{\textrm{deg}(\eilambda{\lambda^{'}}\sigma_w\eilambda{\lambda})}
\end{equation*}
where we have used the above bijection for the third equality. Then we have
\begin{equation*}
\textrm{dim}_q(\textbf{e}\vv\textbf{e})=\sum_{\lambda,\lambda^{'}}q^{\textrm{deg}(\eilambda{\lambda^{'}}\sigma_w\eilambda{\lambda})}\cdot \frac{1}{(1-q^2)^m}=\frac{2^m}{(1-q)^m}
\end{equation*}
and hence dim$_q(\textbf{e}\vv\textbf{e})=$ dim$_q(A^{\otimes(m-1)}\otimes \tilde{A})$ so that $\textbf{e}\vv\textbf{e}\cong A^{\otimes (m-1)}\otimes \tilde{A}$. This, together with the fact that $\textbf{e}$ is full in $\vv$ by Corollary \ref{corollary:e is full in vv}, proves Morita equivalence between $\vv$ and $A^{\otimes (m-1)}\otimes \tilde{A}$.
\end{proof}
We compare this result to a result of Brundan and Kleshchev. Pick $\tilde{\nu} \in \mathbb{N}I$, with multiplicity one, and $|\tilde{\nu}|=m$. Let $\subklr$ be the corresponding KLR algebra.
\begin{theorem} [\cite{Brundan}, Theorem 3.13]
$\subklr$ and $A^{\otimes (m-1)}\otimes_{\textbf{k}} \textbf{k}[x]$ are Morita equivalent.
\end{theorem}
\subsubsection{Affine Quasi-Heredity}
The notion of quasi-heredity for finite-dimensional algebras was first defined by Cline, Parshall and Scott in 1987. The motivating reasons came from the study of highest weight categories arising in the representation theory of semisimple complex Lie algebras and algebraic groups. In 2014, Kleshchev introduced affine quasi-heredity and the definition of an affine highest weight category, see \cite{KleshchevAffineQH}. In particular, he proved an affine analogue of the Cline-Parshall-Scott Theorem; an algebra $A$ is affine quasi-hereditary if and only if the category $A$-mod of finitely generated graded $A$-modules is an affine highest weight category. In \cite{KleshchevAffineQH}, Kleshchev defines these notions in a more general setting, but for our purposes we take $\mathscr{B}$ to be the class of all positively graded polynomial algebras. The definitions that follow are taken from \cite{KleshchevAffineQH}.
\begin{defn}
A graded vector space $V$ is called \textbf{Laurentian} if it is locally finite-dimensional and bounded below. A graded algebra $A$ is Laurentian if it is Laurentian as a graded vector space.
\end{defn}
\begin{lemma}[\cite{KleshchevII}, Lemma 2.2]
Let $H$ be a Laurentian algebra. Then,
\begin{itemize}
\item[(i)] All irreducible $H$-modules are finite-dimensional.
\item[(ii)] $H$ is semiperfect (every finitely generated (graded) $H$-module has a (graded) projective cover); in particular, there are finitely many irreducible $H$-modules up to isomorphism and degree shift.
\end{itemize}
\end{lemma}
\begin{defn}
Let $A$ be a Noetherian Laurentian graded $\textbf{k}$-algebra. $A$ is said to be \textbf{connected} if $A_n=0$ for all $n<0$ and $A_0=\textbf{k}\cdot 1$. 
\end{defn}
For example, all algebras in $\mathscr{B}$ are connected.
\\
\\
For a left Noetherian Laurentian algebra $A$ let 
\begin{equation*}
\{ L(\pi) \mid \pi \in \Pi \}
\end{equation*}
be a complete irredundant set of simple $A$-modules up to isomorphism and degree shift. For each $\pi \in \Pi$, let $P(\pi)$ be the projective cover of $L(\pi)$. That is, $P(\pi)$ is a projective $A$-module and there is a surjection $\theta: P(\pi) \twoheadrightarrow L(\pi)$ with ker$(\theta)$ negligible, i.e. whenever $N \subset P(\pi)$ is a submodule with $N+\textrm{ker}(\theta)=P(\pi)$, then $N=P(\pi)$.\\
\\
We let $q$ be both a formal variable and also a degree shift functor. If $V=\bigoplus_{n \in \mathbb{Z}}V_n$ then $(qV)_n=V_{n-1}$.
For the remainder of this section all algebras will be Noetherian, Laurentian and graded and we will only consider finitely generated modules over these algebras.
\begin{defn} \label{definition:affine quasi-hereditary}
A two-sided ideal $J \subseteq A$ is an \textbf{affine heredity ideal} if
\begin{itemize}
\item[(SI1)] Hom$_A(J,A/J)=0$.
\item[(SI2)] As a left module, $J\cong m(q)P(\pi)$ for some graded multiplicity $m(q) \in \mathbb{Z}[q,q^{-1}]$ and some $\pi \in \Pi$, such that $B_\pi := \textrm{End}_A(P(\pi))^{\textrm{op}} \in \mathscr{B}$.
\item[(PSI)] As a right $B_\pi$-module, $P(\pi)$ is finitely generated and flat.
\end{itemize}
\end{defn}
\begin{remark}
For a connected algebra, a finitely generated module is flat if and only if it is free. For $B_\pi$ connected, the (PSI) condition can be reformulated: as a right $B_\pi$-module, $P(\pi)$ is free finite rank.
\end{remark}
\begin{lemma}[\cite{KleshchevAffineQH}, Lemma 6.5] \label{lemma:from Kleshchev}
Let $J$ be an ideal in the algebra $A$ such that the left $A$-module ${_A}J$ is projective. Then the condition (SI1) is equivalent to the condition $J^2=J$, which in turn is equivalent to $J=AeA$ for an idempotent $e \in A$.
\end{lemma}
\begin{lemma}[\cite{KleshchevAffineQH}, Lemma 6.6] \label{lemma:from Kleshchev II}
Let $J\subseteq A$ be an affine heredity ideal. Write $J=AeA$ for an idempotent $e$, according to Lemma \ref{lemma:from Kleshchev}. Then the natural map $Ae \otimes_{eAe} eA \longrightarrow J$ is an isomorphism. Moreover, we may choose an idempotent $e$ to be primitive so that, using the notation of Definition \ref{definition:affine quasi-hereditary}, we have $Ae \cong P(\pi)$ and $B_\pi \cong eAe$.
\end{lemma}
\begin{defn}
An algebra $A$ is \textbf{affine quasi-hereditary} if there exists a finite chain of ideals 
\begin{equation*}
(0)=J_0\subsetneq J_1 \subsetneq \cdots \subsetneq J_n=A
\end{equation*}
with $J_{i+1}/J_i$ an affine heredity ideal in $A/J_i$, for all $0 \leq i < n$. Such a chain of ideals is called an \textbf{affine heredity chain}.
\end{defn}
\begin{proposition} \label{proposition:tensor product of 2 a.q.h algebras is a.q.h}
If $A$ and $B$ are two affine quasi-hereditary $\textbf{k}$-algebras then the tensor product $A \otimes_{\textbf{k}} B$ is affine quasi-hereditary.
\end{proposition}
\begin{proof}
Let,
\begin{equation} \label{equation:affine heredity chains}
\begin{split}
&(0)=A_0 \subset A_1 \subset \cdots \subset A_n = A \\
&(0)=B_0 \subset B_1 \subset \cdots \subset B_m = B
\end{split}
\end{equation}
be affine heredity chains for $A$ and $B$, respectively. We claim that 
\begin{equation*}
(0) \subset A_1 \otimes B_1 \subset A_1 \otimes B_2 \subset \cdots \subset A_1 \otimes B \subset A_2 \otimes B_1 + A_1 \otimes B \subset A_2 \otimes B_2 + A_1 \otimes B \subset \cdots \subset A \otimes B
\end{equation*}
is an affine heredity chain. We fix $m \geq 1$ and proceed by induction on $n$.
\\
\\
When $n=1$, $A$ has affine heredity chain $(0) = A_0 \subset A_1 = A$ and we claim that $A \otimes B$ has affine heredity chain,
\begin{equation*}
(0) \subset A \otimes B_1 \subset A \otimes B_2 \subset \cdots \subset A \otimes B_m = A \otimes B.
\end{equation*}
We must show that each $(A \otimes B_i)/(A \otimes B_{i-1}) \subset (A \otimes B)/(A \otimes B_{i-1})$, for $1 \leq i \leq m$, is an affine heredity ideal. Note that, $(A \otimes B_i)/(A \otimes B_{i-1}) \cong A \otimes B_i/B_{i-1}$ and $(A \otimes B)/(A \otimes B_{i-1}) \cong A \otimes B/B_{i-1}$. Since $A$ and $B_i/B_{i-1}$ are both idempotent ideals, we know that $A \otimes B_i/B_{i-1}$ is an idempotent ideal. Then, by Lemma \ref{lemma:from Kleshchev}, to show (SI1) it suffices to show that $A \otimes B_i/B_{i-1}$ is projective as a left $A \otimes B/B_{i-1}$-module.
\\
\\
Since $A \subset A$ is an affine heredity ideal we have $A \cong m(q)P(\pi) \in A$-proj, for some graded multiplicity $m(q) \in \mathbb{Z}[q,q^{-1}]$,  such that $B_\pi := \textrm{End}(P(\pi))^{\textrm{op}} \in \mathscr{B}$. By Lemma \ref{lemma:from Kleshchev II}, $P(\pi) \cong Ae$ for some idempotent $e \in A$. Then we have $B_\pi \cong eAe \in \mathscr{B}$.
\\
\\
Similarly, since $B_i/B_{i-1} \subset B/B_{i-1}$ is an affine heredity ideal we have $B_i/B_{i-1} \cong n(q)Q(\sigma) \in B/B_{i-1}$-proj, for some graded multiplicity $n(q) \in \mathbb{Z}[q,q^{-1}]$, such that $B_\sigma := \textrm{End}(Q(\sigma))^{\textrm{op}} \in \mathscr{B}$. Again by Lemma \ref{lemma:from Kleshchev II}, $Q(\sigma) \cong B/B_{i-1} \overbar{f}$ for some idempotent $\overbar{f} \in B/B_{i-1}$. Then we have $B_\sigma \cong \overbar{f}B/B_{i-1}\overbar{f} \in \mathscr{B}$.
\\
\\
Hence, $A \otimes B_i/B_{i-1} \cong m(q)n(q)\cdot P(\pi)\otimes Q(\sigma) \in A \otimes B/B_{i-1}$-proj, so that (SI1) holds. Moreover,
\begin{equation*}
P(\pi) \otimes Q(\sigma) \cong Ae \otimes B/B_{i-1}\overbar{f} = A \otimes B/B_{i-1} \cdot e \otimes \overbar{f}.
\end{equation*}
So $B_{\pi, \sigma}:=\mathrm{End}_{A \otimes B/B_{i-1}}(A \otimes B/B_{i-1} \cdot e \otimes \overbar{f})^{\textrm{op}} \cong eAe \otimes \overbar{f}B/B_{i-1}\overbar{f} \in \mathscr{B}$, so (SI2) also holds. We have so far shown that (SI1) and (SI2) hold. It remains to show (PSI). We know that $P(\pi)$ is free finite rank as a right $B_\pi$-module and that $Q(\sigma)$ is free finite rank as a right $B_\sigma$-module. That is, $P(\pi) \cong (eAe)^k$ and $Q(\sigma) \cong (\overbar{f}B/B_{i-1}\overbar{f})^l$, for some $k,l \in \mathbb{N}$. Hence,
\begin{equation*}
P(\pi) \otimes Q(\sigma) \cong (eAe \otimes \overbar{f}B/B_{i-1}\overbar{f})^{kl}.
\end{equation*}
Then, as a right $B_{\pi, \sigma}$-module, $P(\pi) \otimes Q(\sigma)$ is free finite rank and the statement of the theorem holds for $n=1$.
\\
\\
Suppose now that the result holds whenever $A$ has an affine heredity chain of length $k<n$. Then let $A$ and $B$ be affine quasi-hereditary with affine heredity chains \ref{equation:affine heredity chains}. One can easily verify that 
\begin{equation*}
(0) \subset A_2/A_1 \subset \cdots \subset A_n/A_1 = A/A_1
\end{equation*}
is an affine heredity chain. It is a chain of length $n-1 < n$ so, using induction, the statement of the theorem holds for $(A/A_1) \otimes B$. Hence it suffices to prove that $(A_1 \otimes B_i)/(A_1 \otimes B_{i-1}) \subset (A \otimes B)/(A_1 \otimes B_{i-1})$, for $1 \leq i \leq m$, is an affine heredity ideal. Note that $(A_1 \otimes B_i)/(A_1 \otimes B_{i-1}) \cong A_1 \otimes (B_i/B_{i-1})$. 
\\
\\
$A_1$ and $B_i/B_{i-1}$ are both idempotent ideals so that $A_1 \otimes B_i/B_{i-1}$ is an idempotent ideal. Then, by Lemma \ref{lemma:from Kleshchev}, to show (SI1) it suffices to show that $A_1 \otimes B_i/B_{i-1}$ is a projective $A \otimes B/A_1 \otimes B_{i-1}$-module. Firstly, note that $A_1 \otimes (B/B_{i-1}) \in (A\otimes B)/(A_1 \otimes B_{i-1})$-proj because
\begin{equation*}
A_1 \otimes (B/B_{i-1}) \cong ((A \otimes B)/(A_1 \otimes B_{i-1})) \otimes_{A \otimes B} A_1 \otimes B
\end{equation*}
and, since $A_1 \otimes B \in A \otimes B$-proj, it follows that $A_1 \otimes (B/B_{i-1})$ is a direct summand of a free $(A\otimes B)/(A_1 \otimes B_{i-1})$-module. Then (SI1) is satisfied. We now show (SI2) holds.
\\
\\
Note first that $A_1 \subset A$ is an affine heredity ideal. So 
\begin{equation*}
A_1 \cong m(q)P(\pi) \cong m(q)Ae,
\end{equation*}
for some $\pi \in \Pi$, graded multiplicity $m(q) \in \mathbb{Z}[q,q^{-1}]$, and $e \in A$ a primitive idempotent, where we have used Lemma \ref{lemma:from Kleshchev II}, for the second isomorphism. Similarly, $B_i/B_{i-1} \subset B/B_{i-1}$ is an affine heredity ideal so
\begin{equation*}
B_i/B_{i-1} \cong n(q)Q(\sigma) \cong n(q) B/B_{i-1}\cdot \overbar{f},
\end{equation*}
for some $\sigma \in \Sigma$, graded multiplicity $n(q) \in \mathbb{Z}[q,q^{-1}]$, and $\overbar{f}=f+B_{i-1}$ a primitive idempotent, again using Lemma \ref{lemma:from Kleshchev II}, for the second isomorphism. Then,
\begin{equation*}
\begin{split}
A_1 \otimes B_i/B_{i-1} &\cong m(q)n(q)P(\pi)\otimes Q(\sigma) \\
&\cong m(q)n(q) A \otimes B/B_{i-1}\cdot (e \otimes \overbar{f}) \in (A\otimes B)/(A \otimes B_{i-1})\mathrm{-proj},
\end{split}
\end{equation*}
and is the projective cover of a simple $A \otimes (B/B_i)$-module. However, we need these facts for $(A \otimes B)/(A_1 \otimes B_{i-1})$. There is a short exact sequence of $A \otimes B$-bimodules,
\begin{equation*}
0\longrightarrow ker(g) \longrightarrow (A \otimes B)/(A_1 \otimes B_{i-1}) \overset{g}\longrightarrow (A \otimes B)/(A \otimes B_{i-1}) \longrightarrow 0,
\end{equation*}
where $ker(g) \cong (A \otimes B_{i-1})/(A_1 \otimes B_{i-1}) \cong (A/A_1)\otimes B_{i-1}$. So the short exact sequence is
\begin{equation*}
0\longrightarrow (A/A_1)\otimes B_{i-1} \longrightarrow (A \otimes B)/(A_1 \otimes B_{i-1}) \overset{g}\longrightarrow (A \otimes B)/(A \otimes B_{i-1}) \longrightarrow 0.
\end{equation*}
We now apply the functor $- \otimes_{A \otimes B} (Ae \otimes Bf)$, which we know to be right exact. But $(A/A_1)\otimes B_{i-1}$ vanishes under this functor, since $Ae \subset A_1$. This means we obtain an isomorphism,
\begin{equation*}
(A \otimes B)/(A_1 \otimes B_{i-1})\cdot (e \otimes f) \cong (A \otimes B)/(A \otimes B_{i-1})\cdot (e \otimes f).
\end{equation*}
So,
\begin{equation*}
\begin{split}
A_1 \otimes (B_i/B_{i-1}) &\cong m(q)n(q)(A \otimes B)/(A_1 \otimes B_{i-1})\cdot (e \otimes f) \\
&=m(q)n(q)(A \otimes B)/(A_1 \otimes B_{i-1})\cdot (\overbar{e \otimes f}),
\end{split}
\end{equation*}
which is the projective cover of a simple $(A \otimes B)/(A_1 \otimes B_{i-1})$-module. Let $P(\pi,\sigma):= (A \otimes B)/(A_1 \otimes B_{i-1})\cdot (e \otimes f)$. For (SI2) it remains to check that $B_{\pi,\sigma}:=\mathrm{End}(P(\pi,\sigma))^{\mathrm{op}} \in \mathscr{B}$.
\begin{equation*}
\begin{split}
B_{\pi,\sigma} := \mathrm{End}(P(\pi,\sigma))^{\mathrm{op}} &\cong (e \otimes f)\cdot(A \otimes B)/(A_1 \otimes B_{i-1})\cdot(e \otimes f) \\ &\cong (e \otimes f)\cdot (A\otimes B)/(A \otimes B_{i-1}) \cdot (e \otimes f) \\ &\cong eAe \otimes f(B/B_{i-1})f \in \mathscr{B},
\end{split}
\end{equation*}
since $eAe$, $f(B/B_{i-1})f \in \mathscr{B}$. This proves (SI2). We now show (PSI).
\\
\\
$Ae$ is free finite rank as a right $eAe$-module, i.e. $Ae \cong (eAe)^k$. Similarly, $B/B_{i-1}f$ is free finite rank as a right $f(B/B_{i-1})f$-module, i.e. $(B/B_{i-1})f \cong (f(B/B_{i-1})f)^l$. This implies,
\begin{equation*}
A \otimes (B/B_{i-1}) (e \otimes f) \cong (eAe \otimes f(B/B_{i-1})f)^{kl}.
\end{equation*}
This proves (PSI) so that $(A_1 \otimes B_i)/(A_1 \otimes B_{i-1}) \subset (A \otimes B)/(A_1 \otimes B_{i-1})$, for $1 \leq i \leq m$, is an affine heredity ideal. Thus, using induction, we have shown that $A \otimes B$ is affine quasi-hereditary.
\end{proof}
Let $A$ be the path algebra of the following quiver, as in \ref{subsection:affine cellularity}. 
\begin{center}
\begin{tikzpicture}
\matrix (m) [matrix of math nodes, row sep=3em,
column sep=3em, text height=1.5ex, text depth=0.25ex]
{e_1 & e_2 \\};
\path[->]
(m-1-1) edge [bend left=30] node [above] {$u_1$} (m-1-2)
(m-1-2) edge [bend left=30] node [below] {$u_2$} (m-1-1);
\end{tikzpicture}
\end{center}
$A$ is a graded algebra with deg$(e_i)=0$, $\textrm{deg}(u_i e_i)=1$, for $i\in \{1,2 \}$. In this case, $\Pi=\{1,2\}$. There are two simple graded modules denoted $L(1)$ and $L(2)$, with projective covers $P(1)=Ae_1$ and $P(2)=Ae_2$, respectively. As a \textbf{k}-vector space, $A=\bigoplus_{n \in \mathbb{Z}_{\geq 0}} A_n$, where $A_n$ is the subspace of $A$ spanned by the two degree $n$ elements. So $A$ is a Laurentian algebra. Consider $Ae_1$ as a left $A$-module. Since it is a uniserial module, every submodule is finitely generated and so $Ae_1$ is left Noetherian. Similarly, $Ae_2$ is left Noetherian. Hence $A=Ae_1 \oplus Ae_2$ is left Noetherian, using the fact that the direct sum of two left Noetherian modules is again left Noetherian.
\begin{proposition} \label{proposition:A is a.q.h}
$A$ is affine quasi-hereditary.
\end{proposition}
\begin{proof}
Consider the following chain of ideals
\begin{equation*}
(0) \subsetneq A e_1 A \subsetneq A.
\end{equation*}
We first take $Ae_1A \subsetneq A$ and show it is an affine heredity ideal.
\begin{itemize}
\item[(SI1)] As a \textbf{k}-vector space $A/Ae_1A=\langle e_2 \rangle$. We have $\mathrm{Hom}_A(Ae_1A,A/Ae_1A)=0$; for any $f \in \textrm{Hom}_A(Ae_1A,A/Ae_1A)$ and any $a_1, a_2 \in A$, $f(a_1 e_1 a_2)=a_1e_1f(e_1a_2)=0$, so that Hom$_A(Ae_1A,A/Ae_1A)=0$.
\item[(SI2)] We have $Ae_1A \subseteq Ae_1 + Au_2 e_2$. But clearly $Ae_1 + Au_2 e_2 \subseteq Ae_1A$ and so $Ae_1A = Ae_1 + Au_2 e_2$. Since $Ae_1 \cap Au_2 e_2 = \{0\}$, we have $Ae_1 A = Ae_1 \oplus Au_2 e_2$. So $m_1(q)=1+q$ and $P(1)=Ae_1$. 
\\
\\
$P(1)=Ae_1$ so $B_1= \textrm{End}_A (Ae_1)^{\textrm{op}} \cong e_1 A e_1 \cong \textbf{k}[x]$, a polynomial algebra. So indeed we have $B_1 \in \mathscr{B}$.
\item[(PSI)] As a right $e_1Ae_1$-module, $Ae_1=e_1\cdot e_1Ae_1 + u_1 e_1\cdot e_1 A e_1$ and so is finitely generated. Also, $e_1\cdot e_1Ae_1 \cap u_1 e_1\cdot e_1 A e_1 = \{ 0 \}$. Then, $Ae_1 = \bigoplus_{\{e_1, u_1e_1\}} e_1Ae_1$ as a right $e_1Ae_1$-module and so is projective, which implies $Ae_1$ is flat.
\end{itemize}
This shows that $Ae_1A$ is an affine heredity ideal in $A$. It remains to show that $A/Ae_1A$ is an affine heredity ideal in $A/Ae_1A$. But this is immediate: for any ring $R$, $\textrm{Hom}_R(R,R/R)=0$ so that (SI1) is satisfied. $A/Ae_1A$ has one simple module with $\Pi=\{ 2 \}$. In fact, since
\begin{equation*}
A/Ae_1A \cong \textbf{k}
\end{equation*}
we have that the regular representation $A/Ae_1A$ is a simple $A/Ae_1A$-module with projective cover $P(2)=A/Ae_1A$. So $m_2(q)=1$. Also, $\textrm{End}_{A/Ae_1A}(A/Ae_1A)^{\textrm{op}}\cong A/Ae_1A \cong \textbf{k}$ which lies in $\mathscr{B}$ so (SI2) is satisfied. Since $A/Ae_1A = \textbf{k}e_2=e_2\textbf{k}$, $A/Ae_1A$ is finitely generated as a right \textbf{k}-module and clearly $A/Ae_1A$ is flat. 
\\
\\
Then $(0) \subsetneq A e_1 A \subsetneq A$ is an affine heredity chain and hence $A$ is affine quasi-hereditary as claimed.
\end{proof}
Combining Proposition \ref{proposition:tensor product of 2 a.q.h algebras is a.q.h} with Proposition \ref{proposition:A is a.q.h} we obtain the following Corollary.
\begin{corollary} \label{corollary:m tensor copies of A is affine qh}
The \textbf{k}-algebras $A^{\otimes (m-1)} \otimes_{\textbf{k}} \tilde{A}$, $m \in \mathbb{N}$ are affine quasi-hereditary.
\end{corollary}
\begin{corollary} \label{corollary: VV aff qh when q in I}
For any $\nu \in {^\theta}\mathbb{N}I_q$ with multiplicity one the VV algebras $\vv$ are affine quasi-hereditary.
\end{corollary}
\begin{proof}
Kleshchev proves, in \cite{KleshchevAffineQH}, Theorem 6.7, that an algebra $R$ is affine quasi-hereditary if and only if $R$-Mod satisfies certain properties, in which case he calls $R$-Mod an affine highest weight category. This characterisation of affine quasi-heredity is purely categorical, meaning that affine quasi-heredity is a Morita invariant property. Morita equivalence between $A^{\otimes (m-1)} \otimes \tilde{A}$ and $\vv$, together with Corollary \ref{corollary:m tensor copies of A is affine qh}, proves the claim. 
\end{proof}
\begin{corollary} \label{corollary: VV aff qh when p,q not in I}
Suppose now that $p,q \not\in I$. For any $\nu \in {^\theta}\mathbb{N}I$ the VV algebras $\vv$ are affine quasi-hereditary.
\end{corollary}
\begin{proof}
Using Theorem \ref{theorem:VV Morita equiv to KLR} together with the fact that the algebras  $\subklr$ are affine quasi-hereditary, when $p$ is not a root of unity, (see \cite{KleshchevAffineQH}, Section 10.1) immediately gives the result.
\end{proof}
\subsubsection{Balanced Involution}
Let $A$ be a left Noetherian Laurentian algebra and let $\tau: A \longrightarrow A$ be a homogeneous anti-involution on $A$. That is, $\tau$ is an anti-automorphism of $A$ such that $\tau^2=id_A$. Given a left $A$-module $M \in A$-Mod, we can define a right $A$-module $M^\tau$ via $ma:=\tau(a)m$. Given a graded module $M \in A$-Mod with finite-dimensional graded components $M_n$ we can define its graded dual $M^\circledast \in A$-Mod. As a graded vector space, $M_n^\circledast:=M_{-n}^*$ for all $n \in \mathbb{Z}$, and the action is given by $af(m):=f(\tau(a)m)$, for $f \in M^\circledast$, $m \in M$ and $a \in A$. Note that $(q^nV)^\circledast \cong q^{-n}V^\circledast$ and dim$_qV^\circledast=$dim$_{q^{-1}}V$.
\begin{defn}
The homogeneous anti-involution $\tau$ of $A$ is called a \textbf{balanced involution} if, for every $\pi \in \Pi$, we have that $L(\pi)^\circledast \cong q^nL(\pi)$ for some even integer $n$.
\end{defn}
As before, let $A$ be the path algebra $A=\textbf{k}(e_1 \rightleftarrows e_2)$. Define the following algebra morphism on $A$.
\begin{equation*}
\begin{split}
\tau: A &\longrightarrow A \\
e_i &\mapsto e_i \quad \textrm{ for }i=1,2 \\
u_1 &\mapsto u_2 \\
u_2 &\mapsto u_1.
\end{split}
\end{equation*}
The map $\tau$ is an anti-involution on $A$. Now consider $L(\pi)^\circledast$ for $\pi=1$. 
\begin{equation*}
L(1)^\circledast_n =L(1)^*_{-n}=\left\{
\begin{array}{l l}
			\textrm{Hom}(L(1),\textbf{k}) & \quad \textrm{ if } n=0\\
			0 & \quad \textrm{ if } n\neq 0.
\end{array} \right.
\end{equation*}
As a $\textbf{k}$-vector space Hom$(L(1),\textbf{k})=\langle f \rangle$ where $f(e_1)=1$. Hom$(L(1),\textbf{k})$ is an $A$-module with the action given by $(af)(m)=f(\tau(a)m)$, for all $a \in A$ and $m \in L(1)$. Then, $(e_1f)(e_1)=f(e_1)$ so that $e_1f=f$. $(e_2f)(e_1)=f(e_2e_1)=0$ so $e_2f=0$. Similarly, $u_1f=0=u_2f$. Then Hom$(L(1),\textbf{k}) \cong L(1)$, under the $A$-module isomorphism $f \mapsto e_1$, and hence $L(\pi)^\circledast \cong L(\pi)$ so that $\tau$ is a balanced involution.
\begin{remark}
This also applies to $\tilde{A}$ so we can extend $\tau$ to a homogeneous anti-involution $\tau^{\otimes m}$ on \hbox{$A^{\otimes (m-1)} \otimes_{\textbf{k}} \tilde{A}$}, for any $m \in \mathbb{N}$. Indeed, $\tau^{\otimes m}$ is also a balanced involution.
\end{remark}
Now we have shown the algebras $A^{\otimes (m-1)} \otimes_{\textbf{k}} \tilde{A}$, $m \in \mathbb{N}$, are affine quasi-hereditary we can use the following result of Kleshchev to deduce that $A^{\otimes (m-1)} \otimes_{\textbf{k}} \tilde{A}$, $m \in \mathbb{N}$, are affine cellular.
\begin{proposition}[\cite{KleshchevAffineQH}, Proposition 9.8]
Let $B$ be an affine quasi-hereditary algebra with a balanced involution $\tau$. Then $B$ is an affine cellular algebra.
\end{proposition}
\begin{corollary} \label{corollary:A^otimes (m-1) otimes A is affine cellular}
The algebras $A^{\otimes (m-1)} \otimes_{\textbf{k}} \tilde{A}$, $m \in \mathbb{N}$ are affine cellular with respect to $\tau^{\otimes m}$.
\end{corollary}
\begin{lemma}[\cite{Yang}, Lemma 3.4] \label{lemma:Yang's lemma}
Let $A$ be an algebra with an idempotent $e \in A$ and a \textbf{k}-involution $i$ such that $i(e)=e$. Suppose that $AeA=A$ and $eAe$ is an affine cellular algebra with respect to the restriction of $i$ to $eAe$. Then $A$ is affine cellular with respect to $i$.
\end{lemma}
Let us now recall the following notation. Suppose we have fixed a reduced expression $s_{i_1}\cdots s_{i_k}$ for some $w \in W^B_m$ so that $\sigma_w=\sigma_{i_1}\cdots \sigma_{i_k}$. We have defined $\sigma^\rho_w:=\sigma_{i_k} \cdots \sigma_{i_1}$.
\begin{corollary} \label{corollary:VV algs q in I are affine cellular}
For any $\nu \in {^\theta}\mathbb{N}I_q$ with multiplicity one the VV algebras $\vv$ are affine cellular.
\end{corollary}
\begin{proof}
Fix $\nu \in {^\theta}\mathbb{N}I_q$ with multiplicity one and take the associated VV algebra $\vv$. Fix $\textbf{e}:=\sum_{\lambda \in \Pi^\pm} \eilambda{\lambda}$ and define a $\textbf{k}$-involution on $\vv$ by,
\begin{equation*}
\begin{split}
i:\vv &\longrightarrow \vv \\
x_k &\mapsto x_k \\
\sigma_w &\mapsto \sigma_w^\rho \\
\idemp &\mapsto \idemp
\end{split}
\end{equation*}
for all $1 \leq k \leq m$, $w \in W_m^B$, $\textbf{i}\in {^\theta}I^{\nu}$. Then $i(\textbf{e})=\textbf{e}$ and, from Corollary \ref{corollary:e is full in vv}, $\vv \textbf{e} \vv = \vv$. Using Corollary \ref{corollary:A^otimes (m-1) otimes A is affine cellular} and Theorem \ref{theorem:ME in case q in I mult one}, $\textbf{e} \vv \textbf{e}$ is affine cellular with respect to the restriction of $i$. Then, by Lemma \ref{lemma:Yang's lemma}, $\vv$ is affine cellular with respect to $i$.
\end{proof}
\begin{corollary} \label{corollary:VV are aff cellular when p,q not in I}
The VV algebras $\vv$ in the setting $p,q \notin I$ are affine cellular.
\end{corollary}
\begin{proof}
Fix $\nu \in {^\theta}\mathbb{N}I_\lambda$ and take the associated VV algebra $\vv$. Fix $\textbf{e}:=\sum_{\textbf{i} \in I^{\tilde{\nu}^+}} \idemp$ and let $i$ be the \textbf{k}-involution on $\vv$ from Corollary \ref{corollary:VV algs q in I are affine cellular}. Then $i(\textbf{e})=\textbf{e}$ and, from Theorem \ref{theorem:VV Morita equiv to KLR}, $\vv \textbf{e} \vv = \vv$. From \cite{KleshchevLoubertMiemietz} we know that KLR algebras of type $A$ are affine cellular. This means that $\textbf{e}\vv\textbf{e}$ is affine cellular with respect to the restriction of $i$. We now use Lemma \ref{lemma:Yang's lemma} to conclude that $\vv$ is affine cellular with respect to $i$.
\end{proof}
\addtocontents{toc}{\protect\setcounter{tocdepth}{1}}
\subsection{Morita Equivalence Between $\vv$ and $\subklr^+ \otimes_{\textbf{k}[z]} \tilde{A}$} \label{subsection:ME subsection}
\addtocontents{toc}{\protect\setcounter{tocdepth}{2}}
In this subsection we loosen the restriction on multiplicity by enforcing a multiplicity restriction only on $q$; we take $\nu \in {^\theta}\mathbb{N}I_q$, $|\nu|=2m$, in which $q$ appears with multiplicity one. We also allow $p$ be to a root of unity in this subsection.
\\
\\
Let $\tilde{A}$ be the path algebra of the quiver described in \ref{subsection:affine cellularity}. Then $\tilde{A}$ is a left $\textbf{k}[z]$-module via \begin{equation*}
\begin{split}
z\cdot a_1&=v_2v_1 a_1 \\
z\cdot a_2&=-v_1v_2 a_2.
\end{split}
\end{equation*}
We can also put the structure of a right $\textbf{k}[z]$-module on $\subklr^+$, the action of $z$ on $\subklr^+$ being multiplication by $\sum_{\textbf{i} \in I^{\tilde{\nu}}} x_{\varphi_{\textbf{i}}(q)} \idemp$, where $\varphi_{\textbf{i}}(q)$ denotes the position of $q$ in $\textbf{i}$.
\begin{lemma} \label{lemma:KLR^+ isomorphic to KLR^-}
For every $\nu \in {^\theta}\mathbb{N}I$ with $|\nu|=2m$ there is an isomorphism of KLR algebras, $\subklr^+ \cong \subklr^-$, given by
\begin{equation*}
\begin{split}
\psi: \subklr^+ &\longrightarrow \subklr^- \\
\idemp &\mapsto \etaidemp \\
x_j\idemp &\mapsto -x_{m-j+1}\etaidemp \\
\sigma_k \idemp &\mapsto \sigma_{m-k}\etaidemp
\end{split}
\end{equation*}
and extending \textbf{k}-linearly and multiplicatively, where $\eta \in \mathcal{D}(W^B_m/\sm)$ is the longest element.
\end{lemma}
\begin{proof}
One can use the relations to check that $\psi$ is indeed an algebra morphism and it is clear that this map is a bijection. 
\end{proof}
Let $\textbf{e}:=\sum_{\textbf{i} \in I^{\tilde{\nu}^+}}\idemp + \sum_{\textbf{i} \in I^{\tilde{\nu}^-}}\idemp$. Considering $\textbf{e}\vv\textbf{e}$ yields the following result.
\begin{lemma} \label{lemma:equal dimension formulas}
In the setting described above, $\textrm{dim}_q(\textbf{e} \vv \textbf{e})= \textrm{dim}_q(\subklr^+\otimes_{\textbf{k}[z]} \tilde{A})$. 
\end{lemma}
\begin{proof}
From Lemma \ref{lemma:KLR basis theorem} we have a basis
\begin{equation*}
\mathcal{B} = \lbrace \sigma_{\dot{w}} x_1^{n_1}\cdots x_m^{n_m}\idemp \mid w \in \mathfrak{S}_m, \textbf{i} \in I^{\tilde{\nu}}, n_i \in \mathbb{N}_0 \textrm{ } \forall i\rbrace
\end{equation*}
of $\subklr^+$, and $\tilde{A}$ has basis $\{ a_1, a_2, v_1 a_1, v_2 a_2, v_2v_1 a_1, v_1v_2 a_2, \ldots \}$. Since $za_1 = v_2v_1 a_1$ and $za_2=v_1v_2a_2$, we have the following basis for $\subklr^+ \otimes_{\textbf{k}[z]} \tilde{A}$. 
\begin{equation*}
\mathcal{B}^\prime = \{ b\otimes a_1,b\otimes a_2,b\otimes v_1 a_1, b\otimes v_2 a_2 \mid b \in \mathcal{B}\}.
\end{equation*} 
So the graded dimension of $\subklr^+ \otimes_{\textbf{k}[z]} \tilde{A}$ is 
\begin{equation*}
\begin{split}
\textrm{dim}_q(\subklr^+ \otimes_{\textbf{k}[z]} \tilde{A})&=2\textrm{dim}_q (\subklr^+) + 2q\textrm{dim}_q (\subklr^+) \\
&= 2(1+q)\textrm{dim}_q (\subklr^+).
\end{split}
\end{equation*}
For each $w \in W^B_m$ fix a reduced expression $w=\eta s$, for $\eta \in \mathcal{D}(W^B_m/\sm)$, $s \in \sm$. Take any basis element $\textbf{e} \sigma_w x_1^{n_1}\cdots x_m^{n_m} \textbf{e} \in \textbf{e}\vv\textbf{e}$. Note that 
\begin{equation*}
\textbf{e} \sigma_w x_1^{n_1}\cdots x_m^{n_m} \textbf{e} = \textbf{e} \sigma_\eta \textbf{e} \sigma_s x_1^{n_1}\cdots x_m^{n_m} \textbf{e}.
\end{equation*}
For this expression to be non-zero either $\eta=1$ or $\eta$ is the longest element in $\mathcal{D}(W^B_m/\sm)$. If $\eta=1$ then $\textbf{e}\sigma_\eta\textbf{e}$ has degree 0. If $\eta$ is the longest element in $\mathcal{D}(W^B_m/\sm)$ then $\textbf{e}\sigma_\eta\textbf{e}$ has degree 1. This is because $\sigma_\eta$ starts at an idempotent in $\subklr^+$ and ends at one in $\subklr^-$ so that there is precisely one factor of $\sigma_\eta$ which contributes to the degree of $\sigma_\eta$, namely $\pi\textbf{e}(q^{\pm 1}, \ldots)$. So,
\begin{equation*}
\begin{split}
\textrm{dim}_q(\textbf{e} \vv \textbf{e})&=2\frac{\sum_{\textbf{i} \in I^{\tilde{\nu}}}\big(\sum_{s\in \sm}q^{\textrm{deg}\sigma_s\idemp}+q\sum_{s\in \sm}q^{\textrm{deg}\sigma_s\idemp}\big)}{(1-q^2)^m} \\
&=2\textrm{dim}_q(\subklr^+)+2q\textrm{dim}_q(\subklr^+) \\
&=2(1+q)\textrm{dim}_q(\subklr^+) \\
&=\textrm{dim}_q(\subklr^+ \otimes_{\textbf{k}[z]} \tilde{A})
\end{split}
\end{equation*}
where we have used
\begin{equation*}
\textrm{dim}_q(\subklr^+)=\frac{\sum_{\textbf{i} \in I^{\tilde{\nu}}} \sum_{s\in \sm}q^{\textrm{deg}\sigma_s\idemp}}{(1-q^2)^m}
\end{equation*} 
in the second equality. Hence we have shown $\textrm{dim}_q(\textbf{e} \vv \textbf{e})= \textrm{dim}_q(\subklr^+\otimes_{\textbf{k}[z]} \tilde{A})$.
\end{proof}
\begin{theorem} \label{theorem:W and R x A are ME}
$\vv$ and $\subklr^+ \otimes_{\textbf{k}[z]} \tilde{A}$ are Morita equivalent.
\end{theorem}
\begin{proof}
Let $\textbf{e}:=\sum_{\textbf{i} \in I^{\tilde{\nu}^+}}\idemp + \sum_{\textbf{i} \in I^{\tilde{\nu}^-}}\idemp$ and define a map
\begin{equation*}
\phi: \subklr^+ \otimes_{\textbf{k}[z]} \tilde{A} \longrightarrow \textbf{e}\vv\textbf{e}
\end{equation*}
as,
\begin{equation*}
\begin{split}
\idemp \otimes a_i &\mapsto \left\lbrace  \begin{array}{ll} \idemp \textrm{ if }i=1 \\ \textbf{e}(\eta\textbf{i}) \textrm{ if }i=2 \end{array} \right.  \\
x_j\idemp \otimes a_i &\mapsto \left\lbrace  \begin{array}{ll} x_j\idemp \textrm{ if }i=1 \\ -x_{m-j+1}\etaidemp \textrm{ if }i=2 \end{array} \right.  \\
\sigma_k \idemp \otimes a_i &\mapsto \left\lbrace  \begin{array}{ll} \sigma_k \idemp \textrm{ if }i=1 \\ \sigma_{m-k} \etaidemp \textrm{ if }i=2 \end{array} \right. \\
\idemp \otimes v_i a_i &\mapsto \left\lbrace  \begin{array}{ll}  \sigma_\eta \idemp \textrm{ if }i=1 \\ \sigma_\eta \etaidemp \textrm{ if }i=2, \end{array} \right.
\end{split}
\end{equation*}
for $1\leq j\leq m$, $1\leq k \leq m-1$ and where $\eta \in \mathcal{D}(W^B_m/\sm)$ is the longest element, extending \textbf{k}-linearly and multiplicatively. We must show that $\phi$ is well-defined so that $\phi$ is a morphism of $\textbf{k}$-algebras.
\\
\\
Clearly $\phi$ preserves relations when we restrict the second tensorand to $a_1$. That is, $\phi|_{\subklr^+ \otimes a_1}$ yields an injective morphism $\subklr^+ \hookrightarrow \textbf{e} \vv \textbf{e}$, so the relations are preserved.
\\
\\
Now consider restricting the second tensorand to $a_2$. The restriction of $\phi$ to \hbox{$\subklr^+ \otimes_{\textbf{k}[z]} a_2$} also yields an injection $\subklr^+ \hookrightarrow \textbf{e} \vv \textbf{e}$, with image $\subklr^-$, using Lemma \ref{lemma:KLR^+ isomorphic to KLR^-}. So the relations are also preserved in this case.
\\
\\
We must also show \hbox{$\phi(\idemp\otimes za_1)=\phi(\idemp z \otimes a_1)$} and \hbox{$\phi(\idemp\otimes za_2)=\phi(\idemp z\otimes a_2)$}. Firstly,
\begin{equation*}
\begin{split}
\phi(\idemp\otimes za_1)&=\phi(\idemp \otimes v_2v_1a_1)\\
&=\phi(\idemp\otimes v_2 a_2)\phi(\idemp\otimes v_1 a_1) \\
&=\sigma^\rho_\eta\sigma_\eta\idemp \\
&=x_{\varphi_\textbf{i}(q)}\idemp \\
&=\phi(x_{\varphi_\textbf{i}(q)}\idemp\otimes a_1)\\
&=\phi(\idemp z \otimes a_1).
\end{split}
\end{equation*}
Secondly, noting that $\varphi_{\eta\textbf{i}}(q^{-1})=m-\varphi_{\textbf{i}}(q)+1$,
\begin{equation*}
\begin{split}
\phi(\idemp\otimes za_2)&=\phi(\idemp \otimes -v_1v_2a_2) \\
&=-\phi(\idemp\otimes v_1 a_1)\phi(\idemp\otimes v_2 a_2) \\
&=-\sigma^\rho_\eta\sigma_\eta\etaidemp \\
&=-x_{\varphi_{\eta\textbf{i}}(q^{-1})}\etaidemp \\
&=-x_{m-\varphi_{\textbf{i}}(q)+1}\etaidemp\\
&=\phi(x_{\varphi_\textbf{i}(q)}\idemp\otimes a_2) \\
&=\phi(\idemp z \otimes a_2).
\end{split}
\end{equation*}
The multiplicative identity element in $\subklr^+ \otimes_{\textbf{k}[z]} \tilde{A}$ is $\big(\sum_{\textbf{i} \in I^{\tilde{\nu}^+}} \idemp\big) \otimes (a_1 + a_2)$ and $\textbf{e} \vv \textbf{e}$ has identity element $\textbf{e}$. We have,
\begin{equation*}
\begin{split}
\phi\Big(\big(\sum_{\textbf{i} \in I^{\tilde{\nu}^+}} \idemp\big) \otimes (a_1 + a_2)\Big)&=\sum_{\textbf{i} \in I^{\tilde{\nu}^+}} \phi (\idemp \otimes a_1) + \sum_{\textbf{i} \in I^{\tilde{\nu}^+}} \phi (\idemp \otimes a_2) \\
&=\sum_{\textbf{i} \in I^{\tilde{\nu}^+}} \idemp + \sum_{\textbf{i} \in I^{\tilde{\nu}^+}} \etaidemp \\
&= \sum_{\textbf{i} \in I^{\tilde{\nu}^+}} \idemp + \sum_{\textbf{i} \in I^{\tilde{\nu}^-}} \idemp \\
&= \textbf{e}.
\end{split}
\end{equation*}
So $\phi$ is indeed a $\textbf{k}$-algebra homomorphism. Now to show that $\phi$ is surjective. This is clear for idempotents and for elements $x_j\idemp \in \textbf{e} \vv \textbf{e}$. For every $w \in W^B_m$ fix a reduced expression of the form $w=s \eta$, for $s \in \sm$ and $\eta \in \mathcal{D}(\sm \backslash W^B_m)$ the longest element in the set of minimal length right coset representatives of $\sm$ in $W^B_m$. Take $\idemp \sigma_w \idempj \in \textbf{e} \vv \textbf{e}$, for some $w= s\eta \in W^B_m$.
\begin{itemize}
\item If $\idemp, \idempj \in \subklr^+$ then $\eta=1$ so that $w=s \in \sm$ and $\phi(\sigma_w \idempj \otimes a_1)=\sigma_w\idempj=\idemp\sigma_w\idempj$.
\item Similarly if $\idemp, \idempj \in \subklr^-$ then $\eta=1$ and $w=s \in \sm$. Let $s=s_{i_1}\cdots s_{i_k}$ be a reduced expression for $s$. Note that $\etaidemp \in \subklr^+$ since $\idempj \in \subklr^-$. Then $\sigma_s\idempj=\sigma_{i_1}\cdots\sigma_{i_k}\idempj$ and
\begin{equation*}
\begin{split}
\phi(\sigma_{m-i_1} \cdots \sigma_{m-i_k} \textbf{e} (\eta \textbf{j}) \otimes a_2) &=\sigma_{i_1} \cdots \sigma_{i_k} \idempj \\
&=\sigma_s\idempj.
\end{split}
\end{equation*}
\item Suppose now that $\idemp \in \subklr^+$ and $\idempj \in \subklr^-$. Then $w=s\eta$ and \hbox{$\sigma_w\idempj=\sigma_s \sigma_\eta \idempj$}. Suppose $s$ has fixed reduced expression $s=s_{i_1}\cdots s_{i_k}$. Again note that since $\idempj \in \subklr^-$, we have $\etaidempj \in \subklr^+$. Then
\begin{equation*}
\begin{split}
\phi(\sigma_{i_1}\cdots \sigma_{i_k}\etaidempj \otimes v_2 a_2)&=\phi(\sigma_{i_1}\cdots \sigma_{i_k}\etaidempj \otimes a_1)\phi(\etaidempj\otimes v_2 a_2) \\
&=\sigma_{i_1}\cdots \sigma_{i_k}\etaidempj \sigma_\eta \idempj \\
&=\sigma_s\sigma_\eta\idempj \\
&=\idemp\sigma_w\idempj.
\end{split}
\end{equation*}
The case $\idemp \in \subklr^-$ and $\idempj \in \subklr^+$ is similar.
\end{itemize}
It remains to show that $\phi$ is injective. Since $\phi$ is surjective and, by Lemma \ref{lemma:equal dimension formulas}, $\textrm{dim}_q(\textbf{e} \vv \textbf{e}) = \textrm{dim}_q(\subklr^+ \otimes_{\textbf{k}[z]} \tilde{A})$ injectivity follows immediately and we have a \textbf{k}-algebra isomorphism $\textbf{e}\vv \textbf{e} \cong\subklr^+ \otimes_{\textbf{k}[z]} \tilde{A}$. 
\\
\\
The proof of $\textbf{e}$ full in $\vv$ follows precisely the same reasoning as in the proof of Lemma \ref{lemma:isomorphic idemps}, so we omit this here and instead refer the reader to Lemma \ref{lemma:isomorphic idemps}. Then $\vv\textbf{e}$ is a progenerator  for $\vv$-Mod and we have Morita equivalence between $\vv$ and $\subklr^+ \otimes_{\textbf{k}[z]} \tilde{A}$.
\end{proof}
\subsection{Morita Equivalence in the Case $p \in I$} \label{section:p in I}
In this section we fix the following setting; assume $p \in I$, $q \not\in I$, $p$ is not a root of unity. As noted in \ref{subsection:a note on the cases p,q in I}, we take $\nu \in {^\theta}\mathbb{N}I_p$ so that $p$ always appears with multiplicity greater than 1, i.e. $\nu_p > 1$. In this section we impose a further restriction on our choice of $\nu \in {^\theta}\mathbb{N}I_p$; we pick $\nu$ so that $p$ appears with multiplicity exactly 2, i.e. $\nu_p=2$. As before, $\Pi^+$ denotes the root partitions of $\subklr^+$ and $\Pi^-$ denotes the root partitions of $\subklr^-$. Put $\Pi^\pm :=\Pi^+ \cup\Pi^-$.
\\
\\
As before, let $\tilde{A}$ be the path algebra of the following quiver.
\begin{center}
\begin{tikzpicture}
\matrix (m) [matrix of math nodes, row sep=3em,
column sep=3em, text height=1.5ex, text depth=0.25ex]
{a_1 & a_2 \\};
\path[->]
(m-1-1) edge [bend left=30] node [above] {$v_1$} (m-1-2)
(m-1-2) edge [bend left=30] node [below] {$v_2$} (m-1-1);
\end{tikzpicture}
\end{center}
We define the structure of a left $\textbf{k}[z]$-module on $\tilde{A}$ via 
\begin{equation*}
\begin{split}
z\cdot a_1&=v_2v_1 a_1 \\
z\cdot a_2&=v_1v_2 a_2.
\end{split}
\end{equation*}
We also define the structure of a right $\textbf{k}[z]$-module on $\subklr^+$. The action of $z$ on $\subklr^+$ is multiplication by $\sum_{\textbf{i} \in I^{\tilde{\nu}^+}} (x_{\varphi_{\textbf{i},1}(p)}+x_{\varphi_{\textbf{i},2}(p)}) \idemp$, where $\varphi_{\textbf{i},1}(p)$ denotes the position of the first $p$ appearing in $\textbf{i}$ and $\varphi_{\textbf{i},2}(p)$ denote the position of the second $p$ appearing in $\textbf{i}$. 
\\
\\
Using the definition of ${^\theta}I^\nu$, we note here that for any given $\textbf{i} \in {^\theta}I^\nu$ we have either $p$ appearing twice in $\textbf{i}$, $p^{-1}$ appearing before $p$ in $\textbf{i}$, $p$ appearing before $p^{-1}$ in $\textbf{i}$, or $p^{-1}$ appearing twice in $\textbf{i}$.
\begin{lemma} \label{lemma:p in I, e is full}
Take $\idemp \in \vv$ with $\textbf{i} \notin I^{\tilde{\nu}^+} \cup I^{\tilde{\nu}^-}$.
\begin{itemize}
\item[(i)] If $p$ appears twice in $\textbf{i}$ or if $p^{-1}$ appears before $p$ in $\textbf{i}$ then $\idemp \cong \idempj$ for some $\textbf{j} \in I^{\tilde{\nu}^+}$. 
\item[(ii)] If $p^{-1}$ appears twice in $\textbf{i}$ or if $p$ appears before $p^{-1}$ in $\textbf{i}$ then $\idemp \cong \idempj$ for some $\textbf{j} \in I^{\tilde{\nu}^-}$.
\end{itemize}
\end{lemma}
\begin{proof} 
For $(i)$ assume first that $p$ appears twice in $\textbf{i}$, say $i_{k_1}=i_{k_2}=p$. We require $a \in \idemp \vv \idempj$, $b \in \idempj \vv \idemp$ such that $ab=\idemp$ and $ba=\idempj$, for some $\textbf{j} \in I^{\tilde{\nu}^+}$. 
\\
\\
Since $\textbf{i} \notin I^{\tilde{\nu}^+} \cup I^{\tilde{\nu}^-}$ there is a non-empty subset $\{ \varepsilon_1, \ldots, \varepsilon_d \} \subsetneq \{ 1, \ldots, m \}$, with $\varepsilon_r < \varepsilon_{r+1}$ for all $r$, such that $i_{\varepsilon_r}=p^{-(2n_r+1)}$, where $n_r \in \mathbb{Z}_{>0}$ for all $r$. Note also that $\varepsilon_r \neq k_1, k_2$ for all $r$, since we start with the assumption that $p$ appears twice in $\textbf{i}$. Put,
\begin{equation*}
\begin{split}
&a=\idemp\sigma_{\varepsilon_1-1}\cdots \sigma_1\pi\sigma_{\varepsilon_2-1}\cdots \sigma_1\pi \cdots \sigma_{\varepsilon_d-1} \cdots \sigma_1\pi\idempj \in \idemp\vv\idempj \\
&b=\idempj\pi\sigma_1\cdots \sigma_{\varepsilon_d-1}\pi \sigma_1\cdots \sigma_{\varepsilon_{d-1}-1}\cdots \pi\sigma_1\cdots\sigma_{\varepsilon_1-1}\idemp \in \idempj\vv\idemp.
\end{split}
\end{equation*}
Then $\idempj \in \vv$ is such that $\textbf{j} \in I^{\tilde{\nu}^+}$ and, from the relations, we know that $ab=\idemp$ and $ba=\idempj$ so that $\idemp\cong \idempj$. The same argument can be used for the case when $p^{-1}$ appears before $p$ in $\textbf{i}$. A similar argument is used for $(ii)$.
\end{proof}
\begin{lemma} \label{lemma:a relation in the p in I case}
For every $\textbf{i} \in I^{\tilde{\nu}^+}$ and $\eta \in \mathcal{D}(W^B_m/\sm)$, the longest element, we have
\begin{itemize} 
\item[(i)] $\sigma_\eta^\rho \sigma_\eta\idemp=(x_{\varphi_{\textbf{i},1}(p)}+x_{\varphi_{\textbf{i},2}(p)}) \idemp$
\item[(ii)] $\sigma_\eta^\rho\sigma_\eta \etaidemp=-(x_{\varphi_{\eta\textbf{i},1}(p^{-1})}+x_{\varphi_{\eta\textbf{i},2}(p^{-1})})\etaidemp$.
\end{itemize}
\end{lemma}
\begin{proof}
These equalities are immediate consequences of the relations and the fact that $\Gamma_{I_p}$ has an arrow $p \longrightarrow p^{-1}$.
\end{proof}
\begin{theorem} \label{theorem:ME when p in I}
$\vv$ and $\subklr^+ \otimes_{\textbf{k}[z]} \tilde{A}$ are Morita equivalent.
\end{theorem}
\begin{proof}
Let $\textbf{e}=\sum_{\textbf{i} \in I^{\tilde{\nu}^+}} \idemp + \sum_{\textbf{i} \in I^{\tilde{\nu}^-}} \idemp$.
Define a map 
\begin{equation*}
\phi: \subklr^+ \otimes_{\textbf{k}[z]} \tilde{A} \longrightarrow \textbf{e}\vv\textbf{e}
\end{equation*}
as,
\begin{equation*}
\begin{split}
\idemp \otimes a_i &\mapsto \left\lbrace  \begin{array}{ll} \idemp \textrm{ if }i=1 \\ \textbf{e}(\eta\textbf{i}) \textrm{ if }i=2 \end{array} \right.  \\
x_j\idemp \otimes a_i &\mapsto \left\lbrace  \begin{array}{ll} x_j\idemp \textrm{ if }i=1 \\ -x_{m-j+1}\etaidemp \textrm{ if }i=2 \end{array} \right.  \\
\sigma_k \idemp \otimes a_i &\mapsto \left\lbrace  \begin{array}{ll} \sigma_k \idemp \textrm{ if }i=1 \\ \sigma_{m-k} \etaidemp \textrm{ if }i=2 \end{array} \right. \\
\idemp \otimes v_i a_i &\mapsto \left\lbrace  \begin{array}{ll}  \sigma_\eta \idemp \textrm{ if }i=1 \\ \sigma_\eta \etaidemp \textrm{ if }i=2 \end{array} \right.
\end{split}
\end{equation*}
for $1 \leq j \leq m$, $1 \leq k \leq m-1$ and where $\eta \in \mathcal{D}(W^B_m/\sm)$ is the longest element, extending \textbf{k}-linearly and multiplicatively. We must show that $\phi$ is well-defined so that $\phi$ is a morphism of $\textbf{k}$-algebras.
\\
\\
Firstly, when we restrict the second tensorand to $a_1$ or to $a_2$, $\phi$ preserves the relations using the same argument as in the proof of Theorem \ref{theorem:W and R x A are ME}. It remains to check $\phi(\idemp \otimes za_k) = \phi(\idemp z \otimes a_k)$, for $k=1,2$. We have, using Lemma \ref{lemma:a relation in the p in I case} in the fourth equality,
\begin{equation*}
\begin{split}
\phi(\idemp\otimes za_1)
&=\phi(\idemp \otimes v_2v_1a_1) \\
&=\phi(\idemp \otimes v_2a_2)\phi(\idemp \otimes v_1a_1) \\
&=\sigma^\rho_\eta\sigma_\eta \idemp \\
&=(x_{\varphi_{\textbf{i},1}(p)}+x_{\varphi_{\textbf{i},2}(p)}) \idemp \\
&=\phi((x_{\varphi_{\textbf{i},1}(p)}+x_{\varphi_{\textbf{i},2}(p)}) \idemp \otimes a_1) \\
&=\phi(\idemp z \otimes a_1).
\end{split}
\end{equation*}
Secondly, noting again from Lemma \ref{lemma:a relation in the p in I case} that $\sigma_\eta^\rho\sigma_\eta \etaidemp=-(x_{\varphi_{\eta\textbf{i},1}(p^{-1})}+x_{\varphi_{\eta\textbf{i},2}(p^{-1})})\etaidemp$ and $m-\varphi_{\eta\textbf{i},k}(p^{-1})+1=\varphi_{\textbf{i},k}(p)$ for $k=1,2$ we have
\begin{equation*}
\begin{split}
\phi(\idemp\otimes za_2)
&=\phi(\idemp \otimes v_1v_2a_2) \\
&=\phi(\idemp \otimes v_1a_1)\phi(\idemp \otimes v_2a_2) \\
&=\sigma^\rho_\eta\sigma_\eta \etaidemp \\
&=-(x_{\varphi_{\eta\textbf{i},1}(p^{-1})}+x_{\varphi_{\eta\textbf{i},2}(p^{-1})}) \etaidemp \\
&=\phi((x_{\varphi_{\textbf{i},1}(p)}+x_{\varphi_{\textbf{i},2}(p)}) \idemp \otimes a_2) \\
&=\phi(\idemp z \otimes a_2).
\end{split}
\end{equation*}
The same calculation from the proof of Theorem \ref{theorem:W and R x A are ME} shows
\begin{equation*}
\phi(\big(\sum_{\textbf{i} \in I^{\tilde{\nu}^+}} \idemp\big) \otimes (a_1 + a_2))=\textbf{e}.
\end{equation*} 
Then $\phi$ is indeed a well-defined $\textbf{k}$-algebra morphism.
\\
\\
Now to show that $\phi$ is surjective. This is clear for idempotents $\idemp$ and for $x_j\idemp \in \textbf{e} \vv \textbf{e}$. For every $w \in W^B_m$ fix a reduced expression of the form $w=s \eta$, for $s \in \sm$ and $\eta \in \mathcal{D}(\sm \backslash W^B_m)$ the longest element in the set of minimal length right coset representatives of $\sm$ in $W^B_m$. Take $\idemp \sigma_w \idempj \in \textbf{e} \vv \textbf{e}$, for some $w= s\eta \in W^B_m$.
\begin{itemize}
\item If $\idemp, \idempj \in \subklr^+$ then $\eta=1$ so that $w=s \in \sm$ and $\phi(\sigma_w \idempj \otimes a_1)=\sigma_w\idempj=\idemp\sigma_w\idempj$.
\item Similarly if $\idemp, \idempj \in \subklr^-$ then $\eta=1$ and $w=s \in \sm$. Let $s=s_{i_1}\cdots s_{i_k}$ be a reduced expression for $s$. Then $\sigma_s\idempj=\sigma_{i_1}\cdots\sigma_{i_k}\idempj$ and
\begin{equation*}
\begin{split}
\phi(\sigma_{m-i_1} \cdots \sigma_{m-i_k} \textbf{e} (\eta \textbf{j}) \otimes a_2) &=\sigma_{i_1} \cdots \sigma_{i_k} \idempj \\
&=\sigma_s\idempj.
\end{split}
\end{equation*}
\item Suppose now that $\idemp \in \subklr^+$ and $\idempj \in \subklr^-$. Then $w=s\eta$ and \hbox{$\sigma_w\idempj=\sigma_s \sigma_\eta \idempj$}. Suppose $s$ has fixed reduced expression $s=s_{i_1}\cdots s_{i_k}$. Again note that since $\idempj \in \subklr^-$, we have $\etaidempj \in \subklr^+$. Then
\begin{equation*}
\begin{split}
\phi(\sigma_{i_1}\cdots \sigma_{i_k}\etaidempj \otimes v_2 a_2)&=\phi(\sigma_{i_1}\cdots \sigma_{i_k}\etaidempj \otimes a_1)\phi(\etaidempj\otimes v_2 a_2) \\
&=\sigma_{i_1}\cdots \sigma_{i_k}\etaidempj \sigma_\eta \idempj \\
&=\sigma_s\sigma_\eta\idempj \\
&=\idemp\sigma_w\idempj.
\end{split}
\end{equation*}
The case $\idemp \in \subklr^-$ and $\idempj \in \subklr^+$ is similar.
\end{itemize}
It remains to show that $\phi$ is injective. Since $\phi$ is surjective we can show injectivity by showing equality between the graded dimensions of $\subklr^+\otimes_{\textbf{k}[x]}\tilde{A}$ and $\textbf{e}\vv\textbf{e}$. This follows from the analogue of Lemma \ref{lemma:equal dimension formulas} applied to the current setting. Hence $\phi$ is an isomorphism. 
\\
\\
By Lemma \ref{lemma:p in I, e is full} we have $\textbf{e}$ full in $\vv$. This proves Morita equivalence between $\textbf{e}\vv\textbf{e}$ and $\subklr^+\otimes_{\textbf{k}[x]}\tilde{A}$.
\end{proof}
\subsection{Summary of results.} In various settings in which one defines VV algebras we find a full idempotent $\textbf{e} \in \vv$ and an algebra $S$ such that $\textbf{e}\vv\textbf{e} \cong S$. From this it follows that $\vv$ and $S$ are Morita equivalent, and we call $\textbf{e}$ a Morita idempotent. We are then able to use these Morita equivalences to prove, in certain settings, that VV algebras are affine quasi-hereditary and graded affine cellular. We now summarise the results of Section 2 in the table below. We write `NA' to mean not applicable and we write `?' to indicate that these questions have not yet been answered.
\begin{table}[h]
    \begin{center}
\renewcommand{\arraystretch}{2}
    \begin{tabular}{| >{\centering\arraybackslash}m{1in} | >{\centering\arraybackslash}m{1in} | >{\centering\arraybackslash}m{1in} | >{\centering\arraybackslash}m{1.4in}| >{\centering\arraybackslash}m{1.25in} |}

    \hline

    \textbf{Setting} & $\mathbf{S}$ & \textbf{Morita idempotent} & \textbf{Affine quasi-hereditary} & \textbf{(Graded) affine cellular}  \parbox{0pt}{\rule{0pt}{1ex+\baselineskip}} \\ \hline
    $I\cap J=\emptyset$ & $\w_{\nu_1} \otimes_{\textbf{k}} \w_{\nu_2}$ & $\sum_{\substack{\textbf{i} \in {^\theta}I^{\nu_1} \\ \textbf{j} \in {^\theta}J^{\nu_2}}}\textbf{e}(\textbf{i}\textbf{j})$ & NA & NA \\[0ex] \hline
    $p,q \not\in I$&$\subklr^+$&$\sum_{\textbf{i} \in I^{\tilde{\nu}}}\idemp$&$\checkmark$ (if $p\neq \sqrt[n]{1}$)&$\checkmark$ \\[0ex] \hline
$q \in I$, $p \not\in I$, $p\neq \sqrt[n]{1}$, $\textrm{mult}(\nu)=1$&$A^{\otimes (m-1)} \otimes_{\textbf{k}} \tilde{A}$&$\sum_{\lambda \in \Pi^{\pm}} \textbf{e}(\textbf{i}_\la)$&$\checkmark$&$\checkmark$ \\[0ex] \hline
  $q \in I$, $p \not\in I$, $\textrm{mult}(q)=1$ & $\subklr^+ \otimes_{\textbf{k}[z]} \tilde{A}$ & $\sum_{\textbf{i}\in I^{\tilde{\nu}^{\pm}}}\idemp$ & ? & ? \parbox{0pt}{\rule{0pt}{5ex+\baselineskip}}\\ \hline
 $p \in I$, $q \not\in I$, $p\neq \sqrt[n]{1}$, $\textrm{mult}(p)=2$ & $\subklr^+ \otimes_{\textbf{k}[z]} \tilde{A}$ & $\sum_{\textbf{i}\in I^{\tilde{\nu}^{\pm}}}\idemp$ & ? & ? \parbox{0pt}{\rule{0pt}{5ex+\baselineskip}}\\ \hline

  \end{tabular}
  \end{center}
\end{table}

\bibliography{./MyBibliography.bib}

\providecommand{\bysame}{\leavevmode\hbox to3em{\hrulefill}\thinspace}
\providecommand{\MR}{\relax\ifhmode\unskip\space\fi MR }
\providecommand{\MRhref}[2]{%
  \href{http://www.ams.org/mathscinet-getitem?mr=#1}{#2}
}
\providecommand{\href}[2]{#2}
\begin{thebibliography}{KLM13}

\bibitem[Bru13]{Brundan}
J.~Brundan, \emph{Quiver {H}ecke algebras and categorification}, Advances in
  representation theory of algebras, EMS Ser. Congr. Rep., Eur. Math. Soc.,
  Z\"urich, 2013, pp.~103--133. \MR{3220535}

\bibitem[EK06]{EnomotoKashiwara}
N.~Enomoto and M.~Kashiwara, \emph{{Symmetric crystals and affine Hecke
  algebras of type B}}, Proc. Japan Acad. Ser. A Math. Sci. 82 (2006), no. 8,
  131–136. (2006).

\bibitem[EK09]{EnomotoKashiwara2}
\bysame, \emph{{Symmetric crystals for $\mathfrak{gl}_\infty$}}, Publ. Res.
  Inst. Math. Sci. 44, no. 3, 837–891 (2009).

\bibitem[GL96]{GrahamLehrer}
J.~J. Graham and G.~I. Lehrer, \emph{Cellular algebras}, Invent. Math.
  \textbf{123} (1996), no.~1, 1--34. \MR{1376244 (97h:20016)}

\bibitem[GP00]{GeckPfeiffer}
M.~Geck and G.~Pfeiffer, \emph{{Characters of Finite Coxeter Groups and
  Iwahori-Hecke Algebras}}, Oxford University Press (2000).

\bibitem[KL09]{DiagrammaticApproach}
M.~Khovanov and A.~Lauda, \emph{{A diagrammatic approach to categorification of
  quantum groups I}}, Represent. Theory 13, 309–347. (2009).

\bibitem[Kle]{KleshchevII}
A.~Kleshchev, \emph{Representation theory and cohomology of
  {K}hovanov-{L}auda-{R}ouquier algebras}, preprint arXiv:1401.6151v1.

\bibitem[Kle15]{KleshchevAffineQH}
\bysame, \emph{Affine highest weight categories and affine quasihereditary
  algebras}, Proc. Lond. Math. Soc. (3) \textbf{110} (2015), no.~4, 841--882.
  \MR{3335289}

\bibitem[KLM13]{KleshchevLoubertMiemietz}
A.~Kleshchev, J.~Loubert, and V.~Miemietz, \emph{Affine cellularity of
  {K}hovanov-{L}auda-{R}ouquier algebras in type {$A$}}, J. Lond. Math. Soc.
  (2) \textbf{88} (2013), no.~2, 338--358. \MR{3106725}

\bibitem[KX12]{KoenigXi}
S.~Koenig and C.~Xi, \emph{Affine cellular algebras}, Adv. Math. \textbf{229}
  (2012), no.~1, 139--182. \MR{2854173}

\bibitem[VV11]{VaragnoloVasserot}
M.~Varagnolo and E.~Vasserot, \emph{{Canonical bases and affine Hecke algebras
  of type B}}, Invent. Math. 185, no. 3, 593–693 (2011).

\bibitem[Yan14]{Yang}
G.~Yang, \emph{Affine cellular algebras and {M}orita equivalences}, Arch. Math.
  (Basel) \textbf{102} (2014), no.~4, 319--327. \MR{3196959}

\end{thebibliography}
\bibliographystyle{amsalpha}
\addcontentsline{toc}{section}{References}
\end{document}